\documentclass[reqno, a4paper]{amsart}
\usepackage{amsmath, amssymb, enumerate, esint,bbm}
\usepackage{datetime}
\usepackage[colorlinks=true,linkcolor=blue, citecolor=red]{hyperref}
\usepackage[nameinlink,capitalize]{cleveref}
\usepackage{dsfont}
\usepackage{physics}
\usepackage{braket}
\usepackage{tikz}

\crefname{equation}{}{} 
\numberwithin{equation}{section}

\allowdisplaybreaks

\newtheorem{theo}{Theorem}[section]

\newtheorem{lemm}[theo]{Lemma}

\theoremstyle{definition}
\newtheorem{defi}[theo]{Definition}

\newtheorem{rema}[theo]{Remark}

\newtheorem{assumption}[theo]{Assumption}

\newcommand{\C}{\mathbb C}
\newcommand{\D}{\mathbb D}

\newcommand{\bH}{\mathbb H}
\newcommand{\bN}{\mathbb N}

\newcommand{\R}{\mathbb R}

\newcommand{\Z}{\mathbb Z}

\newcommand{\od}{\mathrm{d}}

\newcommand{\az}{\alpha}

\newcommand{\vz}{\varphi}
\newcommand{\Oz}{\Omega}
\newcommand{\oz}{\omega}
\newcommand{\tz}{\theta}
\newcommand{\Tz}{\Theta}
\newcommand{\ez}{\epsilon}
\newcommand{\gz}{\gamma}
\newcommand{\Gz}{\Gamma}
\newcommand{\sz}{\sigma}

\newcommand{\bz}{\beta}
\newcommand{\dz}{\delta}

\newcommand{\wt}{\widetilde}
\newcommand{\pd}{\partial}

\DeclareMathOperator{\dist}{dist}
\DeclareMathOperator{\diam}{diam}

\newcommand{\fz}{\infty}

\title[Homeomorphic Sobolev extensions]{Homeomorphic Sobolev extensions and integrability of hyperbolic metric}
\author[X. Zhou]{Xilin Zhou}
\address{X. Zhou, Department of Mathematics and Statistics, University of Jyväskylä, P.O. Box 35 (MaD), FI-40014, Finland}
\email{xilin.j.zhou@jyu.fi}

\thanks{This research was supported by the
Finnish centre of excellence in Randomness and Structures of the Academy of Finland, funding decision number: 346305.}

\begin{document}


\begin{abstract}
    Very recently, it was proved that if the hyperbolic metric of a planar Jordan domain is $L^q$-integrable for some $q\in (1,\fz)$,
    then every homeomorphic parametrization of the boundary Jordan curve via the unit circle can be extended to a Sobolev homeomorphism of the entire disk.

    This naturally raises the question of whether the extension holds under more general integrability conditions on the hyperbolic metric.
    
    In this work, we examine the case where the hyperbolic metric is $\phi$-integrable.
    Under appropriate conditions on the function $\phi$,
    we establish the existence of a Sobolev homeomorphic extension for every homeomorphic parametrization of the Jordan curve.
    Moreover, we demonstrate the sharpness of our result by providing an explicit counterexample.
\end{abstract}
\maketitle


\section{Introduction}

Let $\D$ be the unit disk in the complex plane, and $\vz : \pd\D\to\C$ be a (homeomorphic) embedding.
Then $\Oz$, the interior domain bounded by the Jordan curve $\vz(\pd\D)$, is called a Jordan domain, and
$\vz$ is called a parametrization of the Jordan curve $\pd\Oz$.
By the Jordan--Schoenflies theorem, $\vz$ admits a homeomorphic extension from $\overline{\D}$ to $\overline{\Oz}$.
Notably, this theorem requires no assumptions on the regularity of $\vz$ or the boundary $\pd\Oz$.

In variational methods within Geometric Function Theory \cite{AIM2009,HK2014,IM2001,R1989} and Nonlinear Elasticity \cite{A1995,B1976,C1988},
Sobolev regularity of homeomorphisms plays a crucial role.
In Nonlinear Elasticity, the question of Sobolev homeomorphic extension arises in the context of boundary value problems for elastic deformations.
The deformations are minimizers, obtained by taking the infimum of a given elastic energy functional over the space of Sobolev homeomorphisms.
Consequently, a fundamental question is whether the class of admissible homeomorphisms is nonempty—that is, whether a homeomorphic $W^{1,p}$-extension exists.

It is important to note that such extensions may not always exist.
Verchota \cite{V2007} constructed a boundary homeomorphism $\vz:\pd\D\to\pd\D$ that cannot be extended to a Sobolev homeomorphism in $W^{1,2}$. 
Furthermore, Zhang \cite{Z2019} demonstrated the existence of a Jordan domain for which the boundary parametrization,
obtained by Carath\'eodory's theorem from the exterior Riemann map, does not admit a homeomorphic $W^{1,1}$ extension.

Thus, this calls for a geometric condition of the Jordan domain to
guarantee the existence of a homeomorphic $W^{1,p}$-extension for $p\in [1,2)$. Under proper geometric conditions, the problem is well-investigated:
If $\Oz$ is convex, Verchota \cite{V2007} showed that the harmonic extension of any boundary parametrization lies in the Sobolev space $W^{1,p}$, for any $p\in [1,2)$.
A.~Koski and J.~Onninen \cite{KO2021} considered the case when $\pd\Oz$ is rectifiable.
Moreover, P.~Koskela et al \cite{KKO2020} obtained an extension theorem when $\Oz$ is a John domain.

To establish a general geometric condition for the extension problem,
A.~Koski and J.~Onninen \cite{KO2023} derived a criterion based on a series of internal distances within the Jordan domain.
The Gehring--Hayman theorem \cite{GH1962} shows that the internal distance  of a pair of points in the Jordan domain is comparable to the length of the respective hyperbolic geodesic.
Inspired by the above results, O.~Bouchala et al. \cite{BJKXZ2024} obtained a relationship between Sobolev homeomorphic extendability and $L^q$-integrability of the hyperbolic metric $h_{\Oz}$.
\begin{theo}{\cite[Theorem 1.2]{BJKXZ2024}}\label{old_t1}
    Let $\Oz$ be a Jordan domain, $z_0\in\Oz$ and $q\in (1,\fz)$. Let us assume that
    \begin{equation}\label{previous-result}
        \int_{\Oz}\left(h_{\Oz}(z,z_0)\right)^q\,\od z < \fz.
    \end{equation}
    Then each homeomorphic parametrization $\vz:\pd\D\to\pd\Oz$ of our Jordan curve
    has a homeomorphic extension in the class $W^{1,p}(\D,\C)$ for all $p\in [1,2)$.
\end{theo}
Moreover, O.~Bouchala et al. showed the sharpness of the theorem.
\begin{theo}{\cite[Theorem 1.3]{BJKXZ2024}}\label{old_t2}
    There is homeomorphic parametrization $\vz:\pd\D\to\pd\Oz$ of a Jordan curve, so that 
    $\vz$ does not have a homeomorphic $W^{1,1}$-extension, even though 
    \eqref{previous-result} holds for $q = 1$ and some $z_0\in\Oz$.
\end{theo}

It is now natural to inquire whether similar results
can be obtained for refined scales,
such as critical exponents $\az$ for the sufficiency of the
integrability of $h_{\Oz}\cdot\left(\log(e + h_{\Oz})\right)^{\az}$.
For this purpose, we investigate a more general integrability condition for the hyperbolic metric.
\begin{assumption}\label{assume:phi}
    Let $\phi:[0,\fz)\to[0,\fz)$ be a continuous function so that:
    \begin{itemize}
        \item[(1)] $\phi$ is strictly-increasing, and $\lim_{t\to\fz} \phi(t) = \fz$;
        \item[(2)] There is a constant $M > 0$ such that, for all $s,t\in [0,\fz)$,
        \begin{equation}\label{con0}
            \phi(s + t) \leq M\left(\phi(s) + \phi(t)\right);
        \end{equation}
    \end{itemize}
\end{assumption}
Under the above conditions, we prove the main theorem of this paper.
\begin{theo}\label{main-1}
    Let $\Oz$ be a Jordan domain, $z_0\in\Oz$ and suppose
    \begin{equation}\label{con1}
        \int_{\Oz}\phi(h_{\Oz}(z_0,z))\,\od z < \fz,
    \end{equation}
    where $\phi$ satisfies Assumption \ref{assume:phi}.
    If
    \begin{equation}\label{con2}
        \int^{\fz}_1 \frac{1}{\phi(s)}\,\od s < \fz,
    \end{equation}
    then each homeomorphic parametrization $\vz:\pd\D\to\pd\Oz$ of our Jordan curve
    has a homeomorphic extension in the class $W^{1,p}(\D,\C)$ for all $p\in [1,2)$.
\end{theo}
The condition \eqref{con2} in our theorem is sharp. This is shown in the next theorem.
\begin{theo}\label{main-2}
    Let the function $\phi$ satisfy Assumption \ref{assume:phi} and, suppose that
    \begin{equation}\label{con2-counter}
        \int^{\fz}_1 \frac{1}{\phi(s)}\,\od s = \fz.
    \end{equation}
    Suppose that there is a simply connected domain $G$ with $g_0\in G$ so that
    \begin{equation}\label{con2-assume}
        \int_G \phi(h_G(g_0,z))\,\od z < \fz.
    \end{equation}
    Then there exists a Jordan domain $\Oz$ with $z_0\in\Oz$ such that \eqref{con1} holds, while at the same time,
    there exists a homeomorphic parametrization $\vz:\pd\D\to\pd\Oz$ of the Jordan curve
    that does not have a homeomorphic $W^{1,1}$-extension.
\end{theo}
To illustrate the results above, let us consider the functions $\phi_{\az} : [0,\fz)\to [0,\fz)$,
$$\phi_{\az}(t) := t \left(\log(e + t)\right)^{\az},$$
where $\az\in (0,\fz)$. Clearly, \eqref{con2-counter} holds precisely when $\az \in (0,1]$, and \eqref{con2} holds when $\az\in (1,\fz)$.
Furthermore, one may choose $G = \D$ in Theorem \ref{main-2}, see Lemma \ref{lemm:integrability} below.

Theorem \ref{main-2} relies on the following result
according to which even a $W^{1,1}$-extension may fail when the internal diameter $\diam_I$ of the Jordan domain is infinite.
\begin{theo}\label{main-3}
    Let $\Oz$ be a Jordan domain satisfying $\diam_I(\Oz) = \fz$.
    Then, there is a homeomorphic parametrization  $\vz: \pd\D \to \pd\Oz$ of the boundary Jordan curve
    so that $\varphi$ does not have a homeomorphic $W^{1,1}$-extension.
\end{theo}

\section{Sobolev homeomorphic extension, Theorem \ref{main-1}}

In this section, we first recall some fundamental definitions.
We then review the method introduced by A.~Koski and J.~Onninen.
As the main result of this section, we prove Theorem \ref{main-1}.

\medskip

Let $\Oz\subset\C$ be a Jordan domain. For $x,y\in\overline{\Oz}$, the internal distance $d_I(x,y)$ is defined as
$$d_I(x,y) = \inf_{\gz\in\Gz_{x,y}} \ell(\gz),$$
where $\Gz_{x,y}$ denotes the set of all rectifiable curves in $\Oz$ connecting $x$ and $y$, and $\ell(\gz)$ represents the length of the curve $\gz$.
The internal diameter of the domain $\Oz$, denoted by $\diam_I(\Oz)$, is given by
$$\diam_I(\Oz) = \sup_{x,y\in\pd\Oz} d_I(x,y).$$

\medskip

The arc length element for the hyperbolic metric of $\Oz$ is given at $z\in\Oz$ by
$$\od_{S_{hyp}} = \frac{2|g'(z)||\od z|}{1 - |g(z)|^2},$$
where $g$ is the local inverse of the universal covering map $f:\D\to\Oz$.
In our simply connected case, the map $f$ is the conformal homeomorphism obtained by the Riemann mapping theorem.
For $z_1,z_2\in\Oz$, the hyperbolic metric $h_{\Oz}(z_1,z_2)$ is defined by 
$$h_{\Oz}(z_1,z_2) = \inf_{\gz\in \Gz_{z_1,z_2}} \int_{\gz} \,\od_{S_{hyp}}.$$
The infimum is actually achieved by a hyperbolic geodesic. 

The hyperbolic metric is invariant under conformal mappings.
Moreover, let us consider the quasi-hyperbolic metric $k_{\Oz}(z_1,z_2)$ defined by
$$k_{\Oz}(z_1,z_2) = \inf_{\gz\in \Gz_{z_1,z_2}} \int_{\gz} \frac{1}{\dist(z,\pd\Oz)}\,|\od z|.$$
It holds that $k_{\Oz}(z_1,z_2)$ is comparable to $h_{\Oz}(z_1,z_2)$ for all $z_1,z_2\in\Oz$.
For more information on the hyperbolic metric, see \cite{AIM2009}.

\medskip

Next we turn to the method proposed by A.~Koski and J.~Onninen.
We begin by recalling the definition of a dyadic family of closed arcs in $\pd \D$.
Let $\bN := \{0, 1, 2, \dots\}$ and define $\bN^+ := \bN\setminus\{0\}$. 
\begin{defi}
	Let $n_0\in\bN^+$. A family
	$$I := \{I_{n,j}\subset\pd \D : n\geq n_0,\ j = 1,2,3,\ldots,2^n\}$$
    of closed arcs is called dyadic, if for every $n\geq n_0$,  arcs $I_{n,j}$, $j = 1,2,3,\ldots,2^n$, are
	\begin{enumerate}[(i)]
		\item of equal length,
		\item pairwise disjoint apart from their endpoints,
		\item cover $\pd \D$,
	\end{enumerate}
	and, for each arc $I_{n,j}$ there are two arcs in $I$ of half the length of $I_{n,j}$ 
	and so that their union is precisely $I_{n,j}$; 
	these arcs are called the children of $I_{n,j}$ and $I_{n,j}$ is their parent.
\end{defi}

\begin{theo}{\cite[Theorem 3.1]{KO2023}}\label{t0}
	Suppose that $\Omega$ is a Jordan domain and $\varphi\colon\pd \D \to \partial\Omega$ is a boundary homeomorphism.
	Suppose that for some $n_0\in\bN^+$ there is a dyadic family $I = \{I_{n,j}\}$ of closed arcs in $\pd \D$ such that the following hold:
	\begin{enumerate}[1.]
		\item For each $I_{n,j}$ with $n \geq n_0$ there exists a crosscut $\Gamma_{n,j}$ (i.e., a curve in $\Omega$)   connecting 
				   the endpoints of the boundary arc $\varphi(I_{n,j})\subset\partial\Omega$ and such that the estimate
				   \begin{equation}\label{e0}
					\sum^{\infty}_{n = n_0}2^{(p-2)n}\sum^{2^n}_{j = 1}\left(\ell(\Gamma_{n,j})\right)^p<\infty
				   \end{equation}
				   holds. Here $\ell$ stands for the Euclidean length. 
		\item The crosscuts $\Gamma_{n,j}$ for $n\geq n_0$ are all pairwise disjoint apart from their endpoints at the boundary, where $n$ and $j$ are allowed to range over all their possible values.
	\end{enumerate}
	Then $\varphi$ admits a homeomorphic extension from $\overline{\D}$ to $\overline{\Omega}$ in the class $W^{1,p}(\D, \C)$.
\end{theo}

To establish the main theorem,
we prove the following key lemma,
which generalizes \cite[Lemma 2.4]{BJKXZ2024}.

\begin{lemm}\label{lemma:crosscut}
	Suppose that $\Omega$ is a Jordan domain and $f\colon\overline{\D}\to\overline{\Omega}$ is a homeomorphism that is conformal on $\D$.
    Let us assume, further, that $\phi$ satisfies Assumption \ref{assume:phi} and \eqref{con1}.
 Let $\xi_1,\xi_2\in\pd \D$ be such that 
 $$|\xi_1 - \xi_2| \leq \frac{4 \pi}{1 + \pi^2} \approx 1.156,$$
 let $\gamma$ be the hyperbolic geodesic connecting $\xi_1$ and $\xi_2$,
 and $\wt{\gz}$ the shorter-length part of the boundary connecting $\xi_1$ and $\xi_2$.
 Then we have, for $\Gamma := f(\gamma)$ joining $f(\xi_1)$ and $f(\xi_2)$, that
 $$
 		\big(\ell(\Gamma)\big)^2\leq c\left(\int^{\fz}_{\frac{\log 2}{2}} \frac{1}{\phi(s)}\,\od s\right) \int_{\Delta} \phi\left(h_{\Omega}(z, f(0))\right)\od z,
 $$
 	where $\Delta$ is the region bounded by $\Gamma$ and $f(\wt{\gz})$, and $c$ is an absolute positive constant.
\end{lemm}
\begin{proof}
    We follow the same method as in the proof of \cite[Lemma 2.4]{BJKXZ2024}:

    Without loss of generality,
    we assume that the chosen points on the unit circle satisfy
    $\xi_1 = \overline{\xi_2}$, $\Re \xi_1 > 0$ and $\Im \xi_1 > 0$.
	Consider the M\"obius transformation $T$ mapping the unit disk onto the upper half-plane $\bH^+$,
	$$T(z) = a\,\frac{1-z}{1+z},$$
	where $a = \frac{1+\xi_1}{1-\xi_1}=\frac{i\Im\xi_1}{1-\Re\xi_1}$. 
    Then $T$ satisfies the following properties:
    \begin{itemize}
        \item[(1)] $T(\xi_1) = 1 = -T(\xi_2)$, $T(1) = 0$ and $T(-1) = \infty$;
        \item[(2)] $T(\D) = \bH^+$, and $T$ maps the hyperbolic geodesic $\gamma$ connecting $\xi_1$ and $\xi_2$ to the upper half of the unit circle, i.e., $\pd \D\cap \bH^+$. 
    \end{itemize}

    For all $m\in\bN^+$, define $\theta_m := \frac{\pi}{2^m}$, 
    and set 
    $$C_m := \{e^{i\theta}: \theta\in (\theta_{m+1},\theta_m)\}.$$
    For approximating purposes, define $R_m := 1 - \frac{1}{2^{m + 1}}$, $z_m := e^{i\theta_m}$,
    and let
	$$A_m := \left\{re^{i\theta}: r\in (R_m,1]\text{ and }\theta \in (\theta_{m+1},\theta_m)\right\}.$$
    To handle angles between $\frac{\pi}{2}$ and $\pi$, define
    $$A_{-m} := \left\{z\in\D: -\overline{z}\in A_m\right\}, \quad z_{-m} := -\overline{z_m},\quad C_{-m} = \{ z : -\overline{z} \in C_m \}.$$

\tikzset{every picture/.style={line width=0.75pt}} 

\begin{figure}[h]
    \centering
\begin{tikzpicture}[x=0.75pt,y=0.75pt,yscale=-0.73,xscale=0.73]

\draw   (274,171) .. controls (274,138.02) and (300.73,111.29) .. (333.71,111.29) .. controls (366.69,111.29) and (393.43,138.02) .. (393.43,171) .. controls (393.43,203.98) and (366.69,230.71) .. (333.71,230.71) .. controls (300.73,230.71) and (274,203.98) .. (274,171) -- cycle ;
\draw    (235.21,172) -- (430.21,170.02) ;
\draw [shift={(432.21,170)}, rotate = 179.42] [color={rgb, 255:red, 0; green, 0; blue, 0 }  ][line width=0.75]    (10.93,-3.29) .. controls (6.95,-1.4) and (3.31,-0.3) .. (0,0) .. controls (3.31,0.3) and (6.95,1.4) .. (10.93,3.29)   ;
\draw    (334.21,261) -- (333.23,83) ;
\draw [shift={(333.21,81)}, rotate = 89.68] [color={rgb, 255:red, 0; green, 0; blue, 0 }  ][line width=0.75]    (10.93,-3.29) .. controls (6.95,-1.4) and (3.31,-0.3) .. (0,0) .. controls (3.31,0.3) and (6.95,1.4) .. (10.93,3.29)   ;
\draw    (277.71,114.29) .. controls (243.43,85.83) and (216.63,87.51) .. (178.86,116.4) ;
\draw [shift={(177.71,117.29)}, rotate = 322.22] [color={rgb, 255:red, 0; green, 0; blue, 0 }  ][line width=0.75]    (10.93,-3.29) .. controls (6.95,-1.4) and (3.31,-0.3) .. (0,0) .. controls (3.31,0.3) and (6.95,1.4) .. (10.93,3.29)   ;
\draw    (396.71,108.29) .. controls (427.25,82.68) and (462.64,80.35) .. (494.54,104.86) ;
\draw [shift={(496,106)}, rotate = 218.54] [color={rgb, 255:red, 0; green, 0; blue, 0 }  ][line width=0.75]    (10.93,-3.29) .. controls (6.95,-1.4) and (3.31,-0.3) .. (0,0) .. controls (3.31,0.3) and (6.95,1.4) .. (10.93,3.29)   ;
\draw    (392.24,183.33) .. controls (377.09,178.54) and (380.09,158.54) .. (391.24,155.33) ;
\draw    (97.24,180) .. controls (114.24,180) and (136.24,188) .. (145.99,201.25) ;
\draw  [dash pattern={on 0.84pt off 2.51pt}]  (97.24,247) -- (162.24,186) ;
\draw    (162.24,186) -- (145.99,201.25) ;
\draw  [dash pattern={on 0.84pt off 2.51pt}]  (97.24,247) -- (181.09,213.54) ;
\draw    (154.11,193.63) .. controls (163.09,200.54) and (168.09,206.54) .. (171.09,217.54) ;
\draw    (181.09,213.54) -- (171.09,217.54) ;
\draw    (480.09,240.54) .. controls (520.09,210.54) and (580.09,255.54) .. (616,233) ;
\draw    (602.09,115.54) .. controls (636.09,87.54) and (647.09,213.54) .. (616,233) ;
\draw    (533.09,84.54) .. controls (571.09,53.54) and (568.09,141.54) .. (602.09,115.54) ;
\draw    (526.09,230.54) .. controls (538.09,187.54) and (584.09,150.54) .. (633.09,156.54) ;
\draw    (558.09,181.54) .. controls (564.09,194.54) and (574.09,185.54) .. (584.09,180.54) ;
\draw    (584.09,180.54) .. controls (597.09,180.54) and (587.09,171.54) .. (579.09,167.54) ;
\draw    (588.09,173.54) .. controls (599.09,167.54) and (587.09,173.54) .. (598.09,168.54) ;
\draw    (593.09,161.54) .. controls (616.09,161.54) and (599.09,166.54) .. (598.09,168.54) ;
\draw    (0.09,245.54) -- (216.09,244.55) ;
\draw [shift={(218.09,244.54)}, rotate = 179.74] [color={rgb, 255:red, 0; green, 0; blue, 0 }  ][line width=0.75]    (10.93,-3.29) .. controls (6.95,-1.4) and (3.31,-0.3) .. (0,0) .. controls (3.31,0.3) and (6.95,1.4) .. (10.93,3.29)   ;
\draw    (98.09,263.54) -- (98.09,11.54) ;
\draw [shift={(98.09,9.54)}, rotate = 90] [color={rgb, 255:red, 0; green, 0; blue, 0 }  ][line width=0.75]    (10.93,-3.29) .. controls (6.95,-1.4) and (3.31,-0.3) .. (0,0) .. controls (3.31,0.3) and (6.95,1.4) .. (10.93,3.29)   ;
\draw    (9.09,245.54) .. controls (19.09,124.54) and (181.09,134.54) .. (187.09,245.54) ;
\draw    (576.09,203.54) -- (574.25,180.54) ;
\draw [shift={(574.09,178.54)}, rotate = 85.43] [color={rgb, 255:red, 0; green, 0; blue, 0 }  ][line width=0.75]    (10.93,-3.29) .. controls (6.95,-1.4) and (3.31,-0.3) .. (0,0) .. controls (3.31,0.3) and (6.95,1.4) .. (10.93,3.29)   ;
\draw    (606.09,185.54) -- (592.37,169.08) ;
\draw [shift={(591.09,167.54)}, rotate = 50.19] [color={rgb, 255:red, 0; green, 0; blue, 0 }  ][line width=0.75]    (10.93,-3.29) .. controls (6.95,-1.4) and (3.31,-0.3) .. (0,0) .. controls (3.31,0.3) and (6.95,1.4) .. (10.93,3.29)   ;
\draw    (114.09,192.54) -- (127.62,179.91) ;
\draw [shift={(129.09,178.54)}, rotate = 136.97] [color={rgb, 255:red, 0; green, 0; blue, 0 }  ][line width=0.75]    (10.93,-3.29) .. controls (6.95,-1.4) and (3.31,-0.3) .. (0,0) .. controls (3.31,0.3) and (6.95,1.4) .. (10.93,3.29)   ;
\draw    (151.09,207.54) -- (166.2,202.21) ;
\draw [shift={(168.09,201.54)}, rotate = 160.56] [color={rgb, 255:red, 0; green, 0; blue, 0 }  ][line width=0.75]    (10.93,-3.29) .. controls (6.95,-1.4) and (3.31,-0.3) .. (0,0) .. controls (3.31,0.3) and (6.95,1.4) .. (10.93,3.29)   ;
\draw  [dash pattern={on 0.84pt off 2.51pt}]  (177.24,220.86) .. controls (181.24,237.86) and (179.24,223.86) .. (181.24,240.86) ;
\draw  [dash pattern={on 0.84pt off 2.51pt}]  (611.09,162) .. controls (627.09,162) and (624.09,161) .. (630.09,162) ;

\draw (444,63) node [anchor=north west][inner sep=0.75pt]   [align=left] {$f$};
\draw (224,67) node [anchor=north west][inner sep=0.75pt]   [align=left] {$T$};
\draw (394,140) node [anchor=north west][inner sep=0.75pt]   [align=left] {$\xi_1$};
\draw (394.24,186.33) node [anchor=north west][inner sep=0.75pt]   [align=left] {$\xi_2$};
\draw (75,141) node [anchor=north west][inner sep=0.75pt]  [font=\scriptsize] [align=left] {$z_1$};
\draw (166,168) node [anchor=north west][inner sep=0.75pt]  [font=\scriptsize] [align=left] {$z_2$};
\draw (107,198) node [anchor=north west][inner sep=0.75pt]  [font=\tiny] [align=left] {$A_1$};
\draw (65,28) node [anchor=north west][inner sep=0.75pt]  [font=\scriptsize] [align=left] {$T(0)$};
\draw (143,210) node [anchor=north west][inner sep=0.75pt]  [font=\tiny] [align=left] {$A_2$};
\draw (191,211) node [anchor=north west][inner sep=0.75pt]  [font=\scriptsize] [align=left] {$z_3$};
\draw (70,271) node [anchor=north west][inner sep=0.75pt]   [align=left] {$\bH^+$};
\draw (330,269.97) node [anchor=north west][inner sep=0.75pt]   [align=left] {$\D$};
\draw (536,264.97) node [anchor=north west][inner sep=0.75pt]   [align=left] {part of $\Oz$};
\draw (516,241) node [anchor=north west][inner sep=0.75pt]  [font=\scriptsize] [align=left] {$g(-1)$};
\draw (633,135) node [anchor=north west][inner sep=0.75pt]  [font=\scriptsize] [align=left] {$g(1)$};
\draw (530,170) node [anchor=north west][inner sep=0.75pt]  [font=\tiny] [align=left] {$g(z_1)$};
\draw (550,155) node [anchor=north west][inner sep=0.75pt]  [font=\tiny] [align=left] {$g(z_2)$};
\draw (566,205) node [anchor=north west][inner sep=0.75pt]  [font=\tiny] [align=left] {$g(A_1)$};
\draw (2.09,248.54) node [anchor=north west][inner sep=0.75pt]  [font=\scriptsize] [align=left] {$-1$};
\draw (184,246.54) node [anchor=north west][inner sep=0.75pt]  [font=\scriptsize] [align=left] {$1$};
\draw (590,185) node [anchor=north west][inner sep=0.75pt]  [font=\tiny] [align=left] {$g(A_2)$};
\draw (577,145) node [anchor=north west][inner sep=0.75pt]  [font=\tiny] [align=left] {$g(z_3)$};

\end{tikzpicture}
\caption{Transformation between domains}
\end{figure}
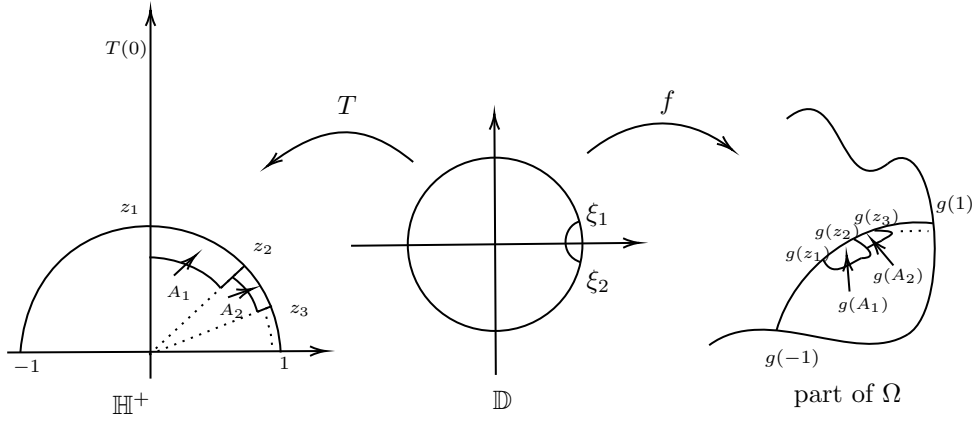

    \medskip

    We set $g := f\circ T^{-1}$.
    By \cite[(2.4), (2.6), and the inequality right after (2.6)]{BJKXZ2024},
    the following results hold: For all $m\in\Z\setminus\{0\}$ and $\oz\in A_m$,
        \begin{equation}\label{lemma:eq01}
            h_{\bH^+}(\omega, T(0)) \geq |m|\log 2,
        \end{equation}
        \begin{equation}\label{lemma:eq02}
            \ell(g(C_m)) \leq \frac{c_1}{2^{|m|}}|g'(z_m)|,
        \end{equation}
        \begin{equation}\label{lemma:eq03}
            \left|g(A_m)\right| \geq \frac{c_2}{2^{2|m|}}|g'(z_m)|^2,
        \end{equation}
    where $c_1$ and $c_2$ are positive constants independent of $m$.
    Then, by \eqref{lemma:eq01}, \eqref{lemma:eq03},
    along with the invariance of the hyperbolic metric under conformal mappings and Assumption \ref{assume:phi}, 
    we conclude that
    \begin{align*}
        \int_{g(A_m)} \phi\left(h_{\Oz}(\oz,f(0))\right)\,\od \oz
        & = \int_{g(A_m)} \phi\left(h_{\bH^+}(g^{-1}(\oz),T(0))\right)\,\od z\\
        & \geq \left|g(A_m)\right| \phi\left(|m|\log 2\right)\\
        & \geq \frac{c_2\phi\left(|m|\log 2\right)}{2^{2|m|}}|g'(z_m)|^2.
    \end{align*}
    Combining this with \eqref{lemma:eq02} and \eqref{con1}, it follows that
    \begin{align*}
        \left(\ell(g(C_m))\right)^2 \leq \frac{(c_1)^2}{c_2}\frac{1}{\phi\left(|m|\log 2\right)}\int_{g(A_m)} \phi\left(h_{\Oz}(\oz,f(0))\right)\,\od \oz.
    \end{align*}
    Thus, by the Cauchy--Schwartz inequality,
    \begin{align*}
        \ell(\Gz) 
        & = \sum_{m\in\Z\setminus\{0\}} \ell(g(C_m))\\
        & \lesssim \sum_{m\in\Z\setminus\{0\}} \left(\frac{1}{\phi\left(|m|\log 2\right)}\int_{g(A_m)} \phi\left(h_{\Oz}(\oz,f(0))\right)\,\od\oz\right)^{1/2}\\
        & \lesssim \left(\sum_{m\in\Z\setminus\{0\}} \frac{1}{\phi\left(|m|\log 2\right)}\right)^{1/2} \left(\sum_{m\in\Z\setminus\{0\}} \int_{g(A_m)} \phi\left(h_{\Oz}(\oz,f(0))\right)\,\od\oz\right)^{1/2}\\
        & \lesssim \left(\sum_{m\in\Z\setminus\{0\}} \frac{1}{\phi\left(|m|\log 2\right)}\right)^{1/2} \left(\int_{\Delta} \phi\left(h_{\Oz}(\oz,f(0))\right)\,\od\oz\right)^{1/2}.
    \end{align*}
    The last inequality follows from the fact that $\{g(A_m)\}_{m\in\Z\setminus\{0\}}$ is pairwise disjoint
    and that $g(A_m)\subset \Delta$.
    Furthermore, by Assumption \ref{assume:phi}, we obtain
    \begin{align*}
        \sum_{m\in\Z\setminus\{0\}} \frac{1}{\phi\left(|m|\log 2\right)}
        & = 2 \sum_{m\in\bN^+}\frac{1}{\phi\left(|m|\log 2\right)}\\
        & \lesssim \frac{1}{\phi(\log 2)} + \sum^{\fz}_{m = 2} \int^{m}_{m - 1} \frac{1}{\phi\left(x\log 2\right)}\,\od x\\
        & \lesssim \int^1_{1/2} \frac{1}{\phi(x\log 2)}\,\od x + \int^{\fz}_1 \frac{1}{\phi(x\log 2)}\,\od x\\
        & \lesssim \int^{\fz}_{\frac{\log 2}{2}}\frac{1}{\phi(x)}\,\od x < \fz.
    \end{align*}
    This completes the proof of Lemma \ref{lemma:crosscut}.
\end{proof}

\smallskip

\begin{rema}
    The estimation in Lemma \ref{lemma:crosscut} is not optimal,
    providing only a uniform upper bound for $\ell(\Gamma)$.
    The result does not offer any insight into the sharpness of the estimate itself or the exponent $\frac{\log 2}{2}$.
\end{rema}

\smallskip

Lemma \ref{lemma:crosscut} leads to a stronger result.
\begin{defi}
	Let $k\in\bN$. A sequence of points $\{x_0,x_1,\ldots,x_k\}\subset\pd \D$ is called a cycle if there exist
	$0 < \theta_1 <\ldots<\theta_k<2\pi$ such that $x_j = e^{i\theta_j}x_0$ for $j=1,...,k$.
    Specifically, set $x_{k + 1} = x_0$.
\end{defi}
\begin{lemm}\label{lemma:internal diameter}
    Let $\Oz$, $f$ and $\phi$ be defined as in Lemma \ref{lemma:crosscut}.
    Given all $k\in\bN^+$ and a cycle $\{x_0,x_1,\ldots,x_k\}\subset\pd \D$ satisfying
    \begin{equation}\label{equ:sizecondition}
        \max\{|x_i - x_{i + 1}|: i = 0,1,\ldots,k\} \leq \frac{4 \pi}{1 + \pi^2},
    \end{equation}
    let $\gamma_i$ be the hyperbolic geodesic connecting $x_i$ and $x_{i + 1}$.
    Then we have, for $\Gamma_i := f(\gamma_i)$ joining $f(x_i)$ and $f(x_{i + 1})$, that
    $$
 		\sum^k_{i = 0}\big(\ell(\Gamma_i)\big)^2\leq c\left(\int^{\fz}_{\frac{\log 2}{2}} \frac{1}{\phi(s)}\,\od s\right) \int_{\Oz} \phi\left(h_{\Omega}(z, f(0))\right)\od z < \fz,
    $$
    where $c$ is a positive constant.
    Furthermore,
    $$\diam_I(\Oz) < \fz.$$
\end{lemm}
\begin{proof}
    Noticing that $\{\Gamma_i\}^k_{i = 0}$ is pairwise disjoint,
    the estimate for $\sum^k_{i = 0}\big(\ell(\Gamma_i)\big)^2$ follows directly from Lemma \ref{lemma:crosscut}.

    It is clear that, there exists a fixed constant $K\in\bN$ such that,
    for all $x,y\in\pd\Oz$, we can find a cycle $\{x_i\}^K_{i = 0}$ satisfying \eqref{equ:sizecondition}, and
    $f(x_l) = x$, $f(x_m) = y$ for some $l,m\in\{0,1,\ldots,K\}$ with $l\leq m$.
    Then, by the Gehring--Hayman inequality and the Cauchy--Schwartz inequality, we obtain
    \begin{align*}
        (d_I(x,y))^2 
        & \leq \left(\sum^m_{i = l} \ell(\Gamma_i)\right)^2\\
        & \leq K\sum^K_{i = 0} \left(\ell(\Gamma_i)\right)^2\\
        & \leq cK\left(\int^{\fz}_{\frac{\log 2}{2}} \frac{1}{\phi(s)}\,\od s\right) \int_{\Oz} \phi\left(h_{\Omega}(z, f(0))\right)\od z < \fz.
    \end{align*}
    This completes the proof of Lemma \ref{lemma:internal diameter}
\end{proof}

\smallskip

We finish this section by proving our main theorem.
\begin{proof}[Proof of Theorem \ref{main-1}]
For $n\in\bN^+$, and let $\Tz_n := \{\tz_{n,j}\}_{j = 0,\ldots,2^n - 1}$ be a cycle on the unit circle satisfying the conditions
\begin{itemize}
    \item[(1)] $\tz_{n,0} = 1$;
    \item[(2)] $\tz_{n,j} = e^{i\frac{2\pi j}{2^n}}\tz_{n,0}$ for all $j = 0,\ldots,2^n - 1$.
\end{itemize}

Let $f:\D\to\Oz$ be the conformal homeomorphism obtained by the Riemann mapping theorem.
By Carath\'eodory's theorem, $f$ extends continuously to a boundary parametrization $f:\pd\D\to\pd\Oz$.
Notice that $\{\Tz_n\}_{n\in\bN^+}$ forms a dense subset of $\pd\D$, and $\Tz_n \subset\Tz_m$ if $n\leq m$.
Thus, there is $n_0\in\bN^+$ such that, for all $n\geq n_0$, 
$$\left|f^{-1}(\vz(\tz_{n,j + 1})) - f^{-1}(\vz(\tz_{n,j}))\right|\leq \frac{4\pi}{1 + \pi^2},$$
for all $j = 0,\ldots,2^n - 1$.

Consider the crosscut $\Gz_{n,j} = f(\gz_{n,j})$ connecting $\vz(\tz_{n,j})$ and $\vz(\tz_{n,j + 1})$,
where $\gz_{n,j}$ is the hyperbolic geodesic connecting $\xi_{n,j} := f^{-1}(\vz(\tz_{n,j}))$ and $\xi_{n, j + 1} := f^{-1}(\vz(\tz_{n,j + 1}))$.
By Lemma \ref{lemma:crosscut}, for all $p\in [1,2)$ and $n\geq n_0$, we obtain
\begin{align*}
    2^{n(p - 2)}\sum^{2^n}_{j = 1} \left(\ell(\Gz_{n,j})\right)^p
    &\lesssim 2^{n(p - 2)}\sum^{2^n}_{j = 1}\left(\int_{\Delta_{n,j}}\phi(h_{\Oz}(z_0,z))\,\od z\right)^{p/2}\\
    &\lesssim 2^{n(p - 2)}2^{n(1 - p/2)} \left(\sum^{2^n}_{j = 1}\int_{\Delta_{n,j}}\phi(h_{\Oz}(z_0,z))\,\od z\right)^{p/2}\\
    &\lesssim 2^{n(p/2 - 1)} \left(\int_{\Oz}\phi(h_{\Oz}(z_0,z))\,\od z\right)^{p/2},
\end{align*}
where $\Delta_{n,j}$ is the region defined in Lemma \ref{lemma:crosscut}
with respect to the points $\xi_{n,j}$ and $\xi_{n,j + 1}$. 
Then, by \eqref{con1}, for all $p \in [1,2)$ we have
$$\sum^{\fz}_{n = n_0} 2^{n(p - 2)}\sum^{2^n}_{j = 1} \left(\ell(\Gz_{n,j})\right)^p \lesssim \left(\sum^{\fz}_{n = n_0}2^{n(p/2 - 1)}\right)\left(\int_{\Oz}\phi(h_{\Oz}(z_0,z))\,\od z\right)^{p/2} < \fz$$
By Theorem \ref{t0}, the boundary parametrization $\vz$
admits a homeomorphic extension from $\overline{\D}$ to $\overline{\Oz}$ in the class $W^{1,p}(\D,\C)$.
This completes the proof of Thenrem \ref{main-1}.
\end{proof}




\section{Counterexample, Theorem \ref{main-2} and Theorem \ref{main-3}}
In this section, we prove Theorem \ref{main-2} and Theorem \ref{main-3}.
By Lemma \ref{lemma:internal diameter}, it follows that
the conditions in Theorem \ref{main-1} imply that the Jordan domain has a finite internal diameter.
We first prove the case when $\diam_I(\Oz) = \fz$ (Theorem \ref{main-3}).
After that, we prove Theorem \ref{main-2} by constructing an explicit example of a Jordan domain with infinite internal diameter.
\begin{proof}[Proof of Theorem \ref{main-3}]
    Given $\oz\in \pd\D$ and $\dz\in [0,\pi]$ define a closed interval on the circle $\pd\D$ by setting
    $$I(\oz,\dz) := \{\oz e^{i\az}: \az\in [-\dz,\dz]\}.$$

    By the Riemann mapping theorem and Carath\'eodory's theorem,
    let $f:\overline{\D}\to\overline{\Oz}$ be the conformal homeomorphism.
    We first show that if $\diam_I(\Oz) = \fz$, 
    then there exists $\oz_0\in\pd\D$ such that, for all $\dz\in (0,\pi)$,
    \begin{equation}\label{equ:limit:point}
        \sup_{\oz\in I(\oz_0,\dz)\setminus\{\oz_0\}} d_I(f(0),f(\oz)) = \fz.
    \end{equation}

    In fact, if $d_I(f(0),f(\oz)) < \fz$ for all $\oz\in\pd\D$,
    there is a sequence $\{\oz_n\}_{n\in\bN^+}\subset\pd\D$ satisfying
    $$d_I(f(0),f(\oz_n)) \geq n, \quad \text{with }\oz_j\neq \oz_k\text{ for }j\neq k.$$
    By the compactness of $\pd\D$,
    there exists $\oz_0\in\pd\D$ such that \eqref{equ:limit:point} holds.

    Otherwise, there is $\wt{\oz}\in\pd\D$ such that $d_I(f(0),f(\wt{\oz})) = \fz$.
    Let $\gz \subset\Oz$ be the internal geodesic connecting $f(0)$ and $f(\wt{\oz})$.
    Then, there exists a sequence $\{x_n\}_{n\in\bN^+} \subset \gz$
    satisfying 
    $$d_I(f(0),x_n) = (3n + 1) \left(\diam(\Oz)\vee 1\right) < \fz$$ 
    for any $n\in\bN^+$. Choose
    $$\eta_n \in \overline{B(x_n,\dist(x_n,\pd\Oz))} \cap \pd\Oz.$$
    By triangle inequality, we have
    \begin{equation}\label{est:di:upper}
        d_I(f(0),\eta_n) \leq d_I(f(0),x_n) + d(x_n,\pd\Oz) < (3n + 2) \left(\diam(\Oz)\vee 1\right) < \fz,
    \end{equation}
    and
    \begin{equation}\label{est:di:lower}
        d_I(f(0),\eta_n) \geq d_I(f(0),x_n) - \dist(x_n,\pd\Oz) \geq 3n\left(\diam(\Oz)\vee 1\right).       
    \end{equation}
    By \eqref{est:di:upper} and \eqref{est:di:lower}, we have $\eta_j\neq\eta_k$ if $j\neq k$. 
    Considering the sequence 
    $\{\oz_n\}_{n\in\bN^+} := \{f^{-1}(\eta_n)\}_{n\in\bN^+}$,
    the compactness of $\pd\D$ ensures the existence of $\oz_0\in\pd\D$ satisfying \eqref{equ:limit:point}.

    \medskip

\tikzset{every picture/.style={line width=0.75pt}} 

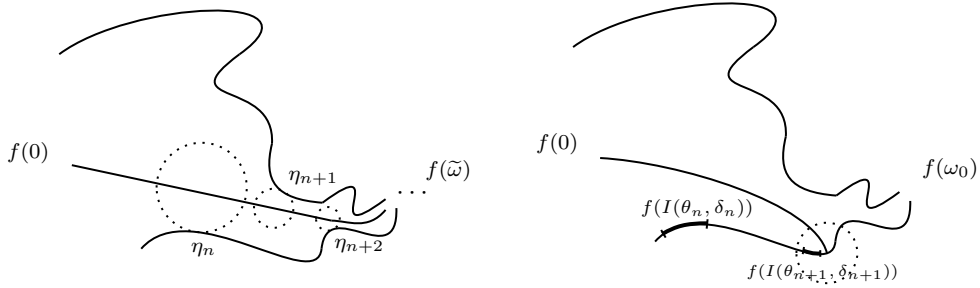
\begin{figure}[h]
    \centering
\begin{tikzpicture}[x=0.75pt,y=0.75pt,yscale=-0.9,xscale=0.9]

\draw    (193.52,172.29) .. controls (188.52,193.29) and (195.52,203.29) .. (221.52,205.29) ;
\draw    (75.52,122.29) .. controls (115.52,92.29) and (224.52,87.29) .. (175.52,122.29) ;
\draw    (193.52,172.29) .. controls (195.52,135.29) and (135.52,152.29) .. (175.52,122.29) ;
\draw    (121,231) .. controls (148.52,199.29) and (218.52,257.29) .. (221,231) ;
\draw    (221.52,205.29) .. controls (261.52,175.29) and (216.52,233.29) .. (256.52,203.29) ;
\draw    (221,231) .. controls (224.52,200.29) and (263.52,245.29) .. (262.52,208.29) ;
\draw    (226.52,215.29) .. controls (235.52,215.29) and (246.52,221.29) .. (256.52,209.29) ;
\draw    (82.52,184.29) -- (183.58,206.04) -- (205.36,210.73) -- (226.52,215.29) ;
\draw  [dash pattern={on 0.84pt off 2.51pt}] (129.52,196.79) .. controls (129.52,182.98) and (140.72,171.79) .. (154.52,171.79) .. controls (168.33,171.79) and (179.52,182.98) .. (179.52,196.79) .. controls (179.52,210.59) and (168.33,221.79) .. (154.52,221.79) .. controls (140.72,221.79) and (129.52,210.59) .. (129.52,196.79) -- cycle ;
\draw  [dash pattern={on 0.84pt off 2.51pt}] (183.56,208.38) .. controls (183.56,202.36) and (188.44,197.47) .. (194.47,197.47) .. controls (200.49,197.47) and (205.38,202.36) .. (205.38,208.38) .. controls (205.38,214.41) and (200.49,219.29) .. (194.47,219.29) .. controls (188.44,219.29) and (183.56,214.41) .. (183.56,208.38) -- cycle ;
\draw  [dash pattern={on 0.84pt off 2.51pt}] (217.96,214.29) .. controls (217.96,210.66) and (220.9,207.72) .. (224.52,207.72) .. controls (228.15,207.72) and (231.09,210.66) .. (231.09,214.29) .. controls (231.09,217.91) and (228.15,220.85) .. (224.52,220.85) .. controls (220.9,220.85) and (217.96,217.91) .. (217.96,214.29) -- cycle ;
\draw    (478.52,168.29) .. controls (473.52,189.29) and (480.52,199.29) .. (506.52,201.29) ;
\draw    (360.52,118.29) .. controls (400.52,88.29) and (509.52,83.29) .. (460.52,118.29) ;
\draw    (478.52,168.29) .. controls (480.52,131.29) and (420.52,148.29) .. (460.52,118.29) ;
\draw    (406,227) .. controls (433.52,195.29) and (503.52,253.29) .. (506,227) ;
\draw    (506.52,201.29) .. controls (546.52,171.29) and (501.52,229.29) .. (541.52,199.29) ;
\draw    (506,227) .. controls (509.52,196.29) and (548.52,241.29) .. (547.52,204.29) ;
\draw [line width=1.5]    (409.99,222.79) .. controls (417.62,218.27) and (423.12,216.77) .. (434.99,216.79) ;
\draw    (409.37,220.52) -- (411.87,224.77) ;
\draw    (435.12,214.52) -- (434.62,219.77) ;
\draw    (375.52,180.29) .. controls (418.04,182.61) and (494.29,205.36) .. (501.29,233.36) ;
\draw  [dash pattern={on 0.84pt off 2.51pt}] (484.39,233.36) .. controls (484.39,224.03) and (491.96,216.46) .. (501.29,216.46) .. controls (510.62,216.46) and (518.18,224.03) .. (518.18,233.36) .. controls (518.18,242.69) and (510.62,250.25) .. (501.29,250.25) .. controls (491.96,250.25) and (484.39,242.69) .. (484.39,233.36) -- cycle ;
\draw [line width=1.5]    (488.39,232.18) .. controls (490.82,232.82) and (493.32,234.07) .. (498.07,233.57) ;
\draw    (488.61,230.04) -- (487.87,234.02) ;
\draw    (497.61,231.79) -- (497.62,235.52) ;

\draw (261,195) node [anchor=north west][inner sep=0.75pt]   [align=left] {$\cdots$};
\draw (45,170) node [anchor=north west][inner sep=0.75pt]   [align=left] {{\footnotesize $f(0)$}};
\draw (277,180) node [anchor=north west][inner sep=0.75pt]  [font=\footnotesize] [align=left] {$f(\wt{\oz})$};
\draw (147,226) node [anchor=north west][inner sep=0.75pt]  [font=\scriptsize] [align=left] {{\footnotesize $\eta_n$}};
\draw (201,186) node [anchor=north west][inner sep=0.75pt]  [font=\scriptsize] [align=left] {{\footnotesize $\eta_{n + 1}$}};
\draw (226.52,223.85) node [anchor=north west][inner sep=0.75pt]  [font=\scriptsize] [align=left] {{\footnotesize $\eta_{n + 2}$}};
\draw (396.38,201.69) node [anchor=north west][inner sep=0.75pt]  [font=\scriptsize] [align=left] {$f(I(\tz_n,\dz_n))$};
\draw (338.36,164.5) node [anchor=north west][inner sep=0.75pt]   [align=left] {{\footnotesize $f(0)$}};
\draw (455,237.69) node [anchor=north west][inner sep=0.75pt]  [font=\tiny] [align=left] {$f(I(\tz_{n + 1},\dz_{n + 1}))$};
\draw (552.5,178.75) node [anchor=north west][inner sep=0.75pt]  [font=\footnotesize] [align=left] {$f(\oz_0)$};

\end{tikzpicture}
\caption{Construction of $\{\eta_n\}$ and $\{I(\tz_n,\dz_n)\}$}
\end{figure}

    \medskip

    Let $\oz_0$ be defined as in \eqref{equ:limit:point}.
    Without loss of generality, we assume that $\oz_0 = e^{i\pi}$.
    We construct the sequence of intervals $\{I(\tz_n,\dz_n)\}_{n\in\bN^+}$ satisfying the following properties
    \begin{itemize}
        \item[(1)] The intervals $\{I(\tz_n,\dz_n)\}_{n\in\bN^+}$ are pairwise disjoint, and $\oz_0\notin I(\tz_n,\dz_n)$;
        \item[(2)] For all $n\in\bN^+$, $\tz_n = e^{it_n}$ where $0 < t_1 < t_2 < \ldots < \pi$, and
        $$\lim_{n\to\fz} t_n = \pi;$$
        \item[(3)] For all $n\in\bN^+$, it holds that
        \begin{equation}\label{equ:interval:lowerbound}
            \inf_{\oz\in I(\tz_n,\dz_n)} d_I(f(0),f(\oz)) \geq 4^n.
        \end{equation}
    \end{itemize}
    The construction of these intervals proceeds in two steps:

    \smallskip

    \textbf{Step 1.}
    By \eqref{equ:limit:point}, we choose $t_1\in (0,\pi)$ such that 
    $$\tz_1 = e^{it_1}\in \pd\D \setminus \{\oz_0\}\quad\text{and}\quad d_I(f(0),f(\tz_1)) \geq 8.$$
    Let $\gz_1$ be the internal geodesic connecting $f(0)$ and $f(\tz_1)$.
    There exists $r_1\in (0,10^{-6})$ and $\eta'_1\in \gz_1\cap \pd B(f(\tz_1),r_1)$ such that
    $$d_I(f(0),\eta'_1) \geq 8 - 10^{-5}.$$
    By the continuity of $f$, there is $\dz_1 \in (0,\fz)$ such that
    $$0 < t_1 - \dz_1 < t_1 + \dz_1 < \pi,$$
    and
    $$f(I(\oz_1,\dz_1)) \subset \pd\Oz \cap B\left(f(\oz_1),r_1\right),$$
    and it is easy to see $\oz_0\notin I(\oz_1,\dz_1)$.
    By triangle inequality, for all $\oz\in I(\oz_1,\dz_1)$, we further conclude that
    $$d_I(f(0),f(\oz)) \geq d_I(f(0),\eta'_1) - 2\pi r_1 \geq 4.$$

    \smallskip

    \textbf{Step 2.}
    Suppose we have already constructed the intervals
    $\{I(\tz_n,\dz_n)\}^k_{n = 1}\subset \pd\D$ satisfying:
    \begin{itemize}
        \item[1.] They are pairwise disjoint and $\oz_0\notin \bigcup^k_{n = 1}I(\tz_n,\dz_n)$;
        \item[2.] There is a sequence $0 < t_1 < \ldots < t_k < t_k + \dz_k < \pi$ such that $\tz_n = e^{it_n}$,
        $$\pi - \frac{\pi}{2^{n - 1}} < t_n - \dz_n < t_n + \dz_n < \pi,$$
        for all $n\in\{1,2,\ldots,k\}$; 
        \item[3.] For all $n\in \{1,2,\ldots,k\}$, \eqref{equ:interval:lowerbound} holds.
    \end{itemize}
    Choose $t_{k + 1}\in (t_k + \dz_k, \pi)\cap (\pi - \frac{\pi}{2^{k}},\pi)$, and $\tz_{k + 1} = e^{it_{k + 1}}$ such that
    $$\tz_{k + 1}\in \pd\D \setminus \left(\{\oz_0\}\cup\bigcup^k_{n = 1} I(\tz_n,\dz_n)\right),$$
    and
    \begin{equation}\label{equ:point:lowerbound}
        d_I(f(0),f(\tz_{k + 1})) \geq 2\times4^{k + 1}.
    \end{equation}
    Since $\oz_0$ is the interior point of $\pd\D\setminus\bigcup^k_{i = 1} I(\tz_i,\dz_i)$,
    the existence of $\tz_{k + 1}$ follows from \eqref{equ:limit:point}.
    By \eqref{equ:point:lowerbound} and using the similar method as in \textbf{Step I},
    for a sufficiently small $r_{k + 1}\in (0,10^{-6})$,
    there exists $\dz_{k + 1} \in (0,\fz)$ such that
    $$\max\left\{\pi - \frac{\pi}{2^k}, t_k + \dz_k\right\} < t_{k + 1} - \dz_{k + 1} < t_{k + 1} + \dz_{k + 1} < \pi,$$ 
    and
    $$f(I(\tz_{k + 1},\dz_{k + 1})) \subset \pd\Oz \cap B(f(\tz_{k + 1}),r_{k + 1}) \setminus f\left(\bigcup^k_{n = 1}I(\tz_n,\dz_n)\right).$$
    Moreover, for all $\oz\in I(\tz_{k + 1},\dz_{k + 1})$, we have
    $$d_I(f(0),f(\oz)) \geq 4^{k + 1}.$$
    It is easy to check that $I(\tz_{k + 1},\dz_{k + 1})$ satisfies all the properties.
    By induction, we obtain the desired sequence of intervals $\{I(\tz_n,\dz_n)\}_{n\in\bN^+}$.

    \smallskip

    Based on these intervals, we construct the boundary parametrization $\vz:\pd\D\to\pd\Oz$ as follows.
    The definition of $\vz$ is given in three cases

    \textbf{Case I.} 
    Set 
    $$\vz(e^{i\pi}) = f(e^{i\pi}).$$

    \textbf{Case II.} 
    For any $n\in\bN^+$, set 
    $$A_n := \left\{e^{i(\az + \bz)}: \az = \pi - \frac{\pi}{2^n}, \bz \in \left[0,\frac{\pi}{4^n}\right]\right\}.$$
    Consider the homeomorphism $g_n: A_n \to I(\tz_n,\dz_n)$ with constant speed (rescaling),
    with
    $$g_n\left(e^{i(\pi - \frac{\pi}{2^n})}\right) = \tz_{n}e^{-i\dz_{n}}\quad\text{and}\quad g_n\left(e^{i(\pi - \frac{\pi}{2^n} + \frac{\pi}{4^n})}\right) = \tz_{n}e^{i\dz_{n}}.$$
    Then, on each $A_n$, define $\vz := f\circ g_n$.

    \textbf{Case III.} 
    It is easy to see that 
    $$\pd\D\setminus \left(\{e^{i\pi}\}\cup\bigcup_{n\in\bN^+}A_n\right)$$ 
    is a countable union of disjoint open intervals on $\pd\D$.
    Define $\vz$ on each interval by rescaling (Similarly as in Case II).
    This completes the construction of $\vz$.

    It is obvious that $\vz$ is injective.
    Notice that $\vz$ is continuous at $e^{i\pi}$,
    and thus the mapping $\vz$ is a boundary parametrization.

    If $\vz$ admits a homeomorphism extension $\varPhi\in W^{1,1}_{loc}$, then
    for some $\ez\in (0,1)$ such that $\D(\varPhi(0),\ez)\subset \Oz$,
    by the definition of internal distance and \eqref{equ:interval:lowerbound},
    it holds that for all $n\in\bN^+$ and $\wt{\oz}\in\pd \D(\varPhi(0),\ez)$
    $$\inf_{\oz\in A_n} d_I(\wt{\oz},\varPhi(\oz)) \gtrsim 4^n,$$
    where the equivalence constant relies on $\ez$.
    We choose $\eta \in (0,\fz)$ such that 
    $$\D(0,\eta)\subset \varPhi^{-1} \left(\D(\varPhi(0),\ez)\right).$$
    Then
    \begin{align*}
        \int_{\D} |D\varPhi(z)|\,\od z
        & \geq \int^{2\pi}_0 \int_{\gz_{t}} |D\varPhi(z)||z|\,\od z\,\od t\\
        & \geq \eta\sum_{n\in\bN^+} \int_{e^{it} \in A_n}\int_{\gz_{t}} |D\varPhi(z)|\,\od z\,\od t\\
        & \gtrsim \eta\sum_{n\in\bN^+} \frac{\pi}{4^n} 4^n \gtrsim \eta\sum_{n\in\bN^+} 1 = \fz,
    \end{align*}
    where $\gz_t$ is the line segement connecting $\eta e^{it}$ and $e^{it}$.
    We conclude that $\varPhi\notin W^{1,1}(\D,\C)$. This completes the proof of Theorem \ref{main-3}.
\end{proof}
\medskip

At the end of this section, we prove Theorem \ref{main-2}.
By Theorem \ref{main-3}, it suffices to construct a Jordan domain which satisfies all the conditions in Theorem \ref{main-2},
and has infinite internal diameter. 
We begin with a technical lemma.
\begin{lemm}\label{lemm:quasilinear}
Let $\phi$ be a function satisfying Assumption \ref{assume:phi}.
Then, for any $\az\in (0,\fz)$, there is a constant $c_{\az,M} \in (0,\fz)$ which depends on $\az$ and $M$, where $M$ is defined as in \eqref{con0},
such that   
    \begin{equation*}
        \phi(\az x) \leq c_{\az,M} \phi(x)
    \end{equation*}
    for any $x\in [0,\fz)$.
\end{lemm}
Lemma \ref{lemm:quasilinear} easily follows from \eqref{con0}, we omit the details here.

\medskip


\begin{lemm}\label{lemm:integrability}
    Given a simply connected domain $\Oz$ and the corresponding Riemann mapping $f: \D \to \Oz$,
    set $z_0 := f(0) \in\Oz$.
    Let $\phi$ be a function satisfying Assumption \ref{assume:phi}.
    If $\phi(h_{\Oz}(z_0,\cdot))\in L^1$, then $\phi(h_{\D}(0,\cdot))\in L^1$.
    Moreover,
    \begin{equation}\label{lemma:phi:equ}
        \int^1_0 \phi\left(\log\left(\frac{1}{1 - r}\right)\right)\,\od r < \fz.
    \end{equation}
\end{lemm}
\begin{proof}
    By a change of variable and the fact that $h_{\D}(0,z)$ depends only on $|z|$, we obtain
    \begin{align*}
        \fz > \int_{\Oz} \phi(h_{\Oz}(z_0,z))\,\od z 
        & = \int_{\D} \phi(h_{\D}(0,z)) |\det Df(z)|\,\od z\\
        & = \int^1_0 \int_{S_r} \phi(h_{\D}(0,z)) |\det Df(z)|\,\od \sz_r(z)\od r\\
        & = \int^1_0 \phi(h_{\D}(0,r)) \int_{S_r} |\det Df(z)|\,\od \sz_r(z)\od r.
    \end{align*}
    The first equality holds due to the invariance of hyperbolic metric under conformal mappings.

    Notice that, for a fixed $r_0\in (0,1)$, there is $\wt{r}_0\in (0,\fz)$ such that
    $B(\oz_0,\wt{r}_0)\subset f(r_0\D)$.
    We conclude that, for any $r\in [r_0,1)$,
    \begin{align*}
        \int_{S_r} |f'(z)|\,\od \sz_r(z) \geq 2\pi \wt{r}_0.
    \end{align*}
    By the Cauchy--Riemann equations, we have $|f'(z)|^2 = |\det Df(z)|$ and hence
    \begin{align*}
        \int_{S_r} |\det Df(z)|\,\od \sz_r(z) \geq \dz,
    \end{align*}
    by the Cauchy--Schwartz inequality where $\dz \in (0,\fz)$ is a constant. Thus, 
    \begin{align*}
        \int_{\D\setminus r_0 \D} \phi(h_{\D}(0,z)) \od z
        & \leq 2\pi \int^1_{r_0} \phi(h_{\D}(0,r))\,\od r\\
        & \leq \frac{2\pi}{\dz} \int_{\Oz} \phi(h_{\Oz}(z_0,z))\,\od z < \fz.
    \end{align*}
    Since $\phi(h_{\D}(0,\cdot))$ is bounded in $r_0\D$, we conclude that $\phi(h_{\D}(0,\cdot))\in L^1$.
    Moreover,
    by Assumption \ref{assume:phi}, we have
    \begin{align*}
        2\pi \int^1_0 r \phi\left(\log\left(\frac{1}{1 - r}\right)\right)\,\od r \leq \int_{\D} \phi\left(h_{\D}(0,z)\right)\,\od z < \fz.
    \end{align*}
    Since $\phi$ is non-decreasing, it follows that \eqref{lemma:phi:equ} holds.
\end{proof}

\medskip

\begin{lemm}\label{lemm:series}
    Let $\{a_n\}^{\fz}_{n = 1}\subset \R^+$.
    If $\sum^{\fz}_{n = 1} a_n = \fz$, then
    \begin{equation}\label{lemm:series:equ}
        \sum^{\fz}_{n = 1} a_n \left(\sum^n_{k = 1}a_k\right)^{\dz - 1}
    \end{equation}
    \begin{itemize}
        \item[(i)] diverges if $\dz\in [0,\fz)$; 
        \item[(ii)] converges if $\dz\in (-\fz,0)$. 
    \end{itemize}
\end{lemm}
\begin{proof}
    Given $n\in\bN^+$, let $S_n := \sum^n_{k = 1}a_k$,
    we have
    \begin{align*}
        \log \left(\sum^{n + 1}_{k = 1} a_n\right) - \log \left(\sum^n_{k = 1} a_n\right)
        & = \log \left(1 + \frac{a_{n + 1}}{S_n}\right)\leq \frac{a_{n + 1}}{S_n}.
    \end{align*}
    Then,
    \begin{align*}
        \sum^n_{k = 1}\frac{a_{k + 1}}{S_k} \geq \log \left(\sum^{n + 1}_{k = 1} a_n\right) - \log a_1 \to\fz\quad\text{as} \quad n\to\fz,
    \end{align*}
    which shows that \eqref{lemm:series:equ} diverges when $\dz = 0$.
    Moreover, notice that $\{a_n\}^{\fz}_{n = 1}$ is positive and $\sum^{\fz}_{n = 1}a_n = \fz$,
    we have, for any $\dz > 0$,
    \begin{align*}
        \sum^{\fz}_{n = 1} a_n \left(\sum^n_{k = 1}a_k\right)^{\dz - 1} \gtrsim  \sum^{\fz}_{n = 1} a_n \left(\sum^n_{k = 1}a_k\right)^{-1} = \fz.
    \end{align*}
    \medskip
    When $\dz\in (-\fz,0)$, using summation by parts and Bernoulli's inequality, we get
    \begin{align*}
        \sum^{n}_{k = 1} a_k \left(\sum^k_{j = 1}a_j\right)^{\dz - 1}
        & = \left(S_n\right)^{\dz - 1}\sum^n_{k = 1} a_n - \sum^{n - 1}_{k = 1}\left(S_{j + 1}\right)^{\dz - 1} - \left(S_j\right)^{\dz - 1}\sum^j_{k = 0}a_k\\
        & = \left(S_n\right)^{\dz} + \sum^{n - 1}_{k = 1}\frac{\left(S_j + a_{j + 1}\right)^{1 - \dz} - \left(S_j\right)^{1 - \dz}}{\left(S_{j + 1}\right)^{1 - \dz}\left(S_j\right)^{1 - \dz}}S_j\\
        & = \left(S_n\right)^{\dz} + \sum^{n - 1}_{k = 1}\frac{\left(1 + \frac{a_{j + 1}}{S_j}\right)^{1 - \dz} - 1}{\left(S_{j + 1}\right)^{1 - \dz}}S_j\\
        & \geq \left(S_n\right)^{\dz} + (1 - \dz)\sum^{n - 1}_{k = 1}\frac{a_{j + 1}}{\left(S_{j + 1}\right)^{1 - \dz}},
    \end{align*}
    then, it holds that
    \begin{align*}
        (a_1)^{\dz} 
        & \geq \left(S_n\right)^{\dz} + (-\dz) \sum^{n}_{k = 2} a_k \left(\sum^k_{j = 1}a_j\right)^{\dz - 1}\\
        & \to \sum^{\fz}_{k = 2} a_k \left(\sum^k_{j = 1}a_j\right)^{\dz - 1}\quad\text{as}\quad n\to\fz.
    \end{align*}
    This finishes the proof.
\end{proof}

\medskip

\begin{proof}[Proof of Theorem \ref{main-2}]
By Assumption \ref{assume:phi} and \eqref{con2-counter}, we obtain
$$\sum_{n\in\bN^+}\frac{1}{\phi(n)} \geq \sum^{\fz}_{n = 2}\int^{n + 1}_n \frac{1}{\phi(x)}\,\od x = \fz$$
For any $n\in\bN^+$, set $a_0 = 0$,
$$a_n := \frac{1}{\phi(n)}\left(\sum^n_{k = 1}\frac{1}{\phi(k)}\right)^{-2/3}\quad\text{and}\quad b_n := \sum^{n - 1}_{k = 0} a_k.$$
Then
\begin{itemize}
    \item[(1)] By Assumption \ref{assume:phi}(1), 
    $a_n$ is non-negative, non-increasing and 
    $$\lim_{n\to\fz} a_n = 0;$$
    \item[(2)] By Assumption \ref{assume:phi}(2), there is a constant $c_M > 1$ such that, for all $n\in\bN^+$,
    \begin{equation}\label{equ:decrease}
        \frac{a_n}{a_{n + 1}} \leq c_M;
    \end{equation}
    \item[(3)] By Lemma \ref{lemm:series}, 
    \begin{equation}\label{equ:convergence}
        \sum_{n\in\bN^+} (a_n)^2 \phi(n) < \fz\quad\text{and}\quad \sum_{n\in\bN^+}a_n = \fz. 
    \end{equation}
\end{itemize}
Consider the domain obtained by gluing by a half disk with a sequence of trapezoids:
\begin{align*}
    R := A \cup \bigcup_{n\in\bN^+} R_n,
\end{align*}
where
\begin{align*}
    &A = \{(x,y)\in\R^2: x\leq 0, |x|^2 + |y|^2 \leq (c_Ma_1)^2\},\\
    &R_n:= \left\{(x + b_n,y)\in\R^2 : 0\leq x \leq a_n, 0\leq |y|\leq c_M\left(a_n - \frac{x(a_n - a_{n + 1})}{a_n}\right)\right\}.
\end{align*}

\tikzset{every picture/.style={line width=0.75pt}} 
\begin{figure}[h]
    \centering
\begin{tikzpicture}[x=0.7pt,y=0.7pt,yscale=-0.8,xscale=0.8]

\draw    (59.09,171.91) -- (581.09,171.91) ;
\draw [shift={(583.09,171.91)}, rotate = 180] [color={rgb, 255:red, 0; green, 0; blue, 0 }  ][line width=0.75]    (10.93,-3.29) .. controls (6.95,-1.4) and (3.31,-0.3) .. (0,0) .. controls (3.31,0.3) and (6.95,1.4) .. (10.93,3.29)   ;
\draw    (213.09,318.91) -- (213.09,22.91) ;
\draw [shift={(213.09,20.91)}, rotate = 90] [color={rgb, 255:red, 0; green, 0; blue, 0 }  ][line width=0.75]    (10.93,-3.29) .. controls (6.95,-1.4) and (3.31,-0.3) .. (0,0) .. controls (3.31,0.3) and (6.95,1.4) .. (10.93,3.29)   ;
\draw  [draw opacity=0] (213.09,297.29) .. controls (141.84,297.29) and (84.09,242.22) .. (84.09,174.29) .. controls (84.09,106.35) and (141.84,51.29) .. (213.09,51.29) -- (213.09,174.29) -- cycle ; \draw   (213.09,297.29) .. controls (141.84,297.29) and (84.09,242.22) .. (84.09,174.29) .. controls (84.09,106.35) and (141.84,51.29) .. (213.09,51.29) ;  
\draw  [dash pattern={on 0.84pt off 2.51pt}]  (339.09,87.91) -- (339.09,262.91) ;
\draw    (213.09,51.29) -- (339.09,87.91) ;
\draw    (213.09,297.29) -- (339.09,262.91) ;
\draw  [dash pattern={on 0.84pt off 2.51pt}]  (213.09,87.91) -- (339.09,87.91) ;
\draw  [dash pattern={on 0.84pt off 2.51pt}]  (215.09,264.91) -- (339.09,262.91) ;
\draw  [dash pattern={on 0.84pt off 2.51pt}]  (428.86,105.91) -- (428.86,247.46) ;
\draw  [dash pattern={on 0.84pt off 2.51pt}]  (213.09,105.91) -- (428.86,105.91) ;
\draw    (339.09,87.91) -- (428.86,105.91) ;
\draw  [dash pattern={on 0.84pt off 2.51pt}]  (212.86,247.46) -- (428.86,247.46) ;
\draw    (339.09,262.91) -- (428.86,247.46) ;
\draw    (428.86,105.91) -- (505.09,113.91) ;
\draw    (428.86,247.46) -- (504.09,237.91) ;
\draw  [dash pattern={on 0.84pt off 2.51pt}]  (213.09,115.91) -- (505.09,113.91) ;
\draw  [dash pattern={on 0.84pt off 2.51pt}]  (213.09,234.91) -- (504.09,237.91) ;
\draw  [dash pattern={on 0.84pt off 2.51pt}]  (505.09,113.91) -- (504.09,237.91) ;

\draw (180,38) node [anchor=north west][inner sep=0.75pt]  [font=\scriptsize] [align=left] {$c_Ma_1$};
\draw (170,301) node [anchor=north west][inner sep=0.75pt]  [font=\scriptsize] [align=left] {$-c_Ma_1$};
\draw (40,174.91) node [anchor=north west][inner sep=0.75pt]  [font=\scriptsize] [align=left] {$-c_Ma_1$};
\draw (180,79) node [anchor=north west][inner sep=0.75pt]  [font=\scriptsize] [align=left] {$c_Ma_2$};
\draw (318.59,174.91) node [anchor=north west][inner sep=0.75pt]  [font=\scriptsize] [align=left] {$b_2$};
\draw (170,260) node [anchor=north west][inner sep=0.75pt]  [font=\scriptsize] [align=left] {$-c_Ma_2$};
\draw (407.59,174.91) node [anchor=north west][inner sep=0.75pt]  [font=\scriptsize] [align=left] {$b_3$};
\draw (180,96.91) node [anchor=north west][inner sep=0.75pt]  [font=\scriptsize] [align=left] {$c_Ma_3$};
\draw (170,242.91) node [anchor=north west][inner sep=0.75pt]  [font=\scriptsize] [align=left] {$-c_Ma_3$};
\draw (145,139) node [anchor=north west][inner sep=0.75pt]   [align=left] {$A$};
\draw (260,138) node [anchor=north west][inner sep=0.75pt]   [align=left] {$R_1$};
\draw (373,138) node [anchor=north west][inner sep=0.75pt]   [align=left] {$R_2$};
\draw (455,136) node [anchor=north west][inner sep=0.75pt]   [align=left] {$R_3$};
\draw (483.59,173.91) node [anchor=north west][inner sep=0.75pt]  [font=\scriptsize] [align=left] {$b_4$};
\draw (180,111.91) node [anchor=north west][inner sep=0.75pt]  [font=\scriptsize] [align=left] {$c_Ma_4$};
\draw (170,226.91) node [anchor=north west][inner sep=0.75pt]  [font=\scriptsize] [align=left] {$-c_Ma_4$};
\draw (535,133) node [anchor=north west][inner sep=0.75pt]   [align=left] {$\cdots$};
\draw (535,186) node [anchor=north west][inner sep=0.75pt]   [align=left] {$\cdots$};

\end{tikzpicture}
\caption{The domain $R$}
\end{figure}
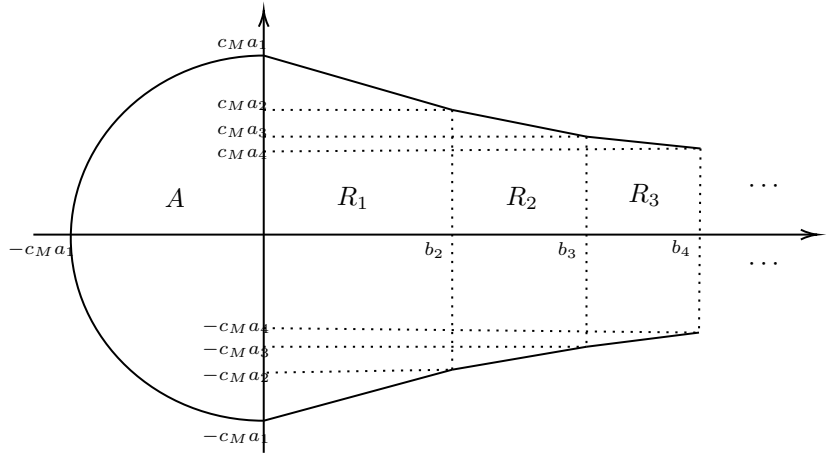

\medskip


Set $r_n(s) := c_M\left(a_n - \frac{s(a_n - a_{n + 1})}{a_n}\right)$ for all $s\in [0,a_n]$.
Recall that $k_{R}$ denotes the quasi-hyperbolic metric of the domain $R$.
By \eqref{equ:decrease}, it holds that
\begin{align*}
    k_R(0,(s + b_n,t))
    & \leq k_R(0,(s + b_n,0)) + k_R((s + b_n,0),(s + b_n,t))\\
    & \lesssim \sum^{n}_{k = 1} \frac{a_k}{c_Ma_{k + 1}} + \int^t_0 \frac{1}{r_n(s) - r}\,\od r
    \lesssim n + \log\left(\frac{r_n(s)}{r_n(s) - t}\right),
\end{align*}
where the equivalent constants do not rely on $n$, $s$ and $t$.
Then, by Lemma \ref{lemm:quasilinear} and Lemma \ref{lemm:integrability}, we obtain
\begin{align*}
    \int_{R_n} \phi(k_R(0,z))\,\od z
    &\lesssim |R_n| \phi(n) + \int^{a_n}_0\int^{r_n(s)}_{-r_n(s)}\phi\left(\log \frac{r_n(s)}{r_n(s) - |t|}\right)\,\od t\,\od s\\
    &\lesssim (a_n)^2 \phi(n) + \left(\int^{a_n}_0 r_n(s)\,\od s\right)\left(\int^1_0 \phi\left(\log \frac{1}{1 - |t|}\right)\,\od t\right)\\
    &\lesssim (a_n)^2 \phi(n) + (a_n)^2 \left(\int^1_0 \phi\left(\log \frac{1}{1 - t}\right)\,\od t\right).
\end{align*}
Combining this with \eqref{equ:convergence} and Lemma \ref{lemm:integrability}, we have
\begin{align}\label{equ:quasihyper:intergrability}
    \int_R \phi(k_R(0,z))\,\od z 
    & = \int_{A} \phi(k_R(0,z))\,\od z + \sum_{n\in\bN^+}\int_{R_n} \phi(k_R(0,z))\,\od z\notag\\
    & \lesssim \int_{A} \phi(k_R(0,z))\,\od z \notag\\
    & \quad \qquad+ \left(\sum_{n\in\bN^+} (a_n)^2 \phi(n)\right)\left(1 + \int^1_0 \phi\left(\log \frac{1}{1 - t}\right)\,\od t\right) < \fz. 
\end{align}

\smallskip

On the other hand, because $\sum_{n\in\bN^+} (a_n)^2 < \fz$, 
and $a_n$ is positive and decreasing,
without loss of generality, assume that
$$\sum_{n\in\bN^+} (a_n)^2 = \sum_{n\in\bN^+}\frac{1}{4^n}.$$
For any $n\in\bN^+$, set $i_n$ the maximal integer such that
$$\sum^{\fz}_{k = i_n} (a_k)^2 \geq \sum^{\fz}_{k = n + 1}\frac{1}{4^k}.$$
Then, the sequence $\{i_n\}_{n\in\bN^+}$ is non-decreasing, and, for all $n\in\bN^+$ such that $i_{n + 1} > i_n$,
\begin{equation}\label{equ:category}
    \sum^{i_{n + 1}}_{k = i_{n} + 1} (a_k)^2 \leq \frac{1}{4^{n - 1}}.
\end{equation}

We assume, for convenience, 
that the sequence $\{i_n\}_{n\in\bN^+}$ is strictly increasing.
Set $l_n = \frac{1}{2^n}$. 
We aim to fold the union $\bigcup^{i_{n + 1}}_{k = i_n + 1} R_n$
into a rectangle $Q_n$ with height $8c_Ml_n$ and width $w_n$.
If the series of $\sum w_n$ converges, then the desired Jordan domain is obtained.

The construction proceeds as follows:
For a fixed $n\in\bN^+$, set $m_1 := i_n + 1$. Let 
$m_2$ be the minimal integer such that $m_1 < m_2 < i_{n + 1} - 1$ and
$$\frac{l_n}{2} \leq a_{m_1} + a_{m_1 + 1} + \cdots a_{m_2 - 1}\leq 2 l_n.$$
Next, let $s_1$ be the smallest non-negative integer satisfying
\begin{equation}\label{equ:overlap}
    3c_M a_{m_2} + a_{m_2} + a_{m_2 + 1} + \cdots + a_{m_2 + s_1 - 1} \geq c_M a_{m_1},
\end{equation}
and let $m_3$ be the minimal integer such that $m_2 + s_1 \leq  m_3 < i_{n + 1} - 1$ with
\begin{equation}\label{equ:enoughlength}
    \frac{l_n}{2} \leq a_{m_2 + s_1} + a_{m_1 + s_1 + 1} + \cdots a_{m_3 - 1}\leq 2 l_n.
\end{equation}

By induction, there exists an integer $K_n\in\bN^+$ and a sequence of integers $\{m_d\}^{K_n}_{d = 1}$ such that
$$i_n + 1 = m_1 \leq m_2 \leq m_2 + s_1 \leq m_3 \leq m_3 + s_2 \leq m_4 \leq \cdots \leq m_{K_n} = i_{n + 1}.$$
Note that $m_{K_n}$ does not necessarily satisfy \eqref{equ:enoughlength}.
The folding domain is as follows:

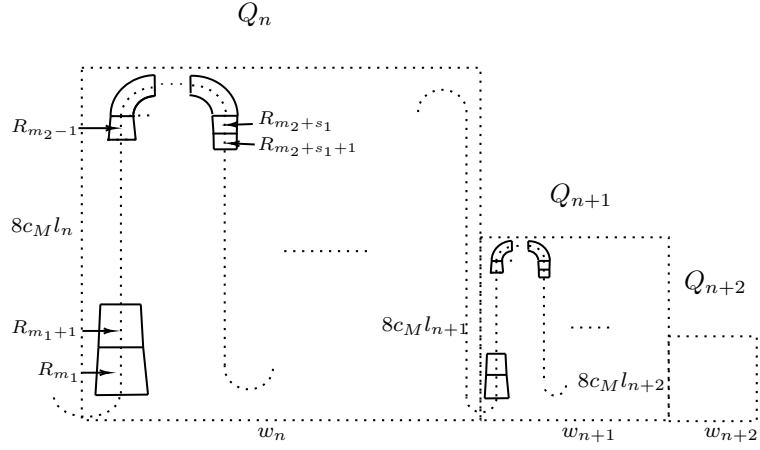
\begin{figure}[h]
\centering
\begin{tikzpicture}[x=0.75pt,y=0.75pt,yscale=-1,xscale=1]
    
    \draw  [dash pattern={on 0.84pt off 2.51pt}] (91.82,87.32) -- (290.82,87.32) -- (290.82,264.61) -- (91.82,264.61) -- cycle ;
    \draw  [dash pattern={on 0.84pt off 2.51pt}] (384.73,222.41) -- (428.73,222.41) -- (428.73,265.09) -- (384.73,265.09) -- cycle ;
    \draw [line width=0.75]    (128.18,91.02) -- (128.32,101.76) ;
    \draw    (101.08,206.35) -- (100.12,227.97) ;
    \draw    (122.46,227.93) -- (124.87,251.72) ;
    \draw  [dash pattern={on 0.84pt off 2.51pt}]  (111.35,111.54) -- (111.34,251.71) ;
    \draw  [dash pattern={on 0.84pt off 2.51pt}]  (104.55,123.53) -- (119.19,123.37) ;
    \draw    (117.52,111.76) -- (118.71,123.51) ;
    \draw  [dash pattern={on 0.84pt off 2.51pt}]  (106.52,111.67) -- (128.85,111.2) ;
    \draw    (100.12,227.97) -- (122.66,227.93) ;
    \draw    (100.12,227.97) -- (99.43,238.99) -- (98.61,252) ;
    \draw    (101.08,206.35) -- (121.25,206.36) ;
    \draw    (121.25,206.36) -- (122.46,227.93) ;
    \draw    (106.21,111.69) -- (105.16,124.01) ;
    \draw  [draw opacity=0] (106.21,111.69) .. controls (106.2,111.5) and (106.2,111.31) .. (106.2,111.13) .. controls (106.21,100.21) and (115.99,91.34) .. (128.18,91.02) -- (128.85,111.15) -- cycle ; \draw   (106.21,111.69) .. controls (106.2,111.5) and (106.2,111.31) .. (106.2,111.13) .. controls (106.21,100.21) and (115.99,91.34) .. (128.18,91.02) ;  
    \draw  [draw opacity=0][dash pattern={on 0.84pt off 2.51pt}] (111.35,111.54) .. controls (111.88,103.86) and (118.67,97.72) .. (127.15,97.25) -- (128.21,112.5) -- cycle ; \draw  [dash pattern={on 0.84pt off 2.51pt}] (111.35,111.54) .. controls (111.88,103.86) and (118.67,97.72) .. (127.15,97.25) ;  
    \draw  [draw opacity=0] (117.52,111.78) .. controls (117.52,111.69) and (117.51,111.6) .. (117.51,111.51) .. controls (117.52,106.06) and (122.27,101.63) .. (128.14,101.58) -- (128.24,111.52) -- cycle ; \draw   (117.52,111.78) .. controls (117.52,111.69) and (117.51,111.6) .. (117.51,111.51) .. controls (117.52,106.06) and (122.27,101.63) .. (128.14,101.58) ;  
    \draw  [dash pattern={on 0.84pt off 2.51pt}]  (163.12,111) -- (163.17,237.22) ;
    \draw    (156.96,111.49) -- (157.56,120.21) ;
    \draw    (169.5,111.13) -- (169.18,120.05) ;
    \draw    (157.56,120.21) -- (157.71,128.35) ;
    \draw    (169.18,120.05) -- (169.18,128.35) ;
    \draw  [draw opacity=0][dash pattern={on 0.84pt off 2.51pt}] (187.24,238.75) .. controls (186.92,244.48) and (181.64,249.04) .. (175.18,249.04) .. controls (168.51,249.03) and (163.12,244.16) .. (163.12,238.17) .. controls (163.12,237.85) and (163.14,237.54) .. (163.17,237.22) -- (175.19,238.18) -- cycle ; \draw  [dash pattern={on 0.84pt off 2.51pt}] (187.24,238.75) .. controls (186.92,244.48) and (181.64,249.04) .. (175.18,249.04) .. controls (168.51,249.03) and (163.12,244.16) .. (163.12,238.17) .. controls (163.12,237.85) and (163.14,237.54) .. (163.17,237.22) ;  
    \draw    (98.61,252) -- (124.87,251.72) ;
    \draw    (104.55,123.53) -- (118.71,123.51) ;
    \draw    (157,120.41) -- (169.73,120.41) ;
    \draw    (157.71,128.35) -- (169.18,128.35) ;
    \draw    (156.96,111.49) -- (169.7,111.49) ;
    \draw    (106.52,111.67) -- (118.25,111.67) ;
    \draw  [dash pattern={on 0.84pt off 2.51pt}]  (192.86,179.57) -- (234.86,179.57) ;
    \draw    (92.14,240.71) -- (103.14,240.71) ;
    \draw [shift={(105.14,240.71)}, rotate = 180] [color={rgb, 255:red, 0; green, 0; blue, 0 }  ][line width=0.75]    (4.37,-1.32) .. controls (2.78,-0.56) and (1.32,-0.12) .. (0,0) .. controls (1.32,0.12) and (2.78,0.56) .. (4.37,1.32)   ;
    \draw    (92.14,219.71) -- (105.14,219.71) ;
    \draw [shift={(107.14,219.71)}, rotate = 180] [color={rgb, 255:red, 0; green, 0; blue, 0 }  ][line width=0.75]    (4.37,-1.32) .. controls (2.78,-0.56) and (1.32,-0.12) .. (0,0) .. controls (1.32,0.12) and (2.78,0.56) .. (4.37,1.32)   ;
    \draw    (90.14,117.71) -- (106.14,117.71) ;
    \draw [shift={(108.14,117.71)}, rotate = 180] [color={rgb, 255:red, 0; green, 0; blue, 0 }  ][line width=0.75]    (4.37,-1.32) .. controls (2.78,-0.56) and (1.32,-0.12) .. (0,0) .. controls (1.32,0.12) and (2.78,0.56) .. (4.37,1.32)   ;
    \draw    (177.18,115.62) -- (167.18,115.98) ;
    \draw [shift={(165.18,116.05)}, rotate = 357.95] [color={rgb, 255:red, 0; green, 0; blue, 0 }  ][line width=0.75]    (4.37,-1.32) .. controls (2.78,-0.56) and (1.32,-0.12) .. (0,0) .. controls (1.32,0.12) and (2.78,0.56) .. (4.37,1.32)   ;
    \draw    (178.18,125.92) -- (168.18,125.44) ;
    \draw [shift={(166.18,125.35)}, rotate = 2.73] [color={rgb, 255:red, 0; green, 0; blue, 0 }  ][line width=0.75]    (4.37,-1.32) .. controls (2.78,-0.56) and (1.32,-0.12) .. (0,0) .. controls (1.32,0.12) and (2.78,0.56) .. (4.37,1.32)   ;
    \draw  [draw opacity=0] (146.03,91.37) .. controls (146.34,91.36) and (146.66,91.35) .. (146.98,91.35) .. controls (159.46,91.37) and (169.57,100.34) .. (169.62,111.42) -- (146.96,111.49) -- cycle ; \draw   (146.03,91.37) .. controls (146.34,91.36) and (146.66,91.35) .. (146.98,91.35) .. controls (159.46,91.37) and (169.57,100.34) .. (169.62,111.42) ;  
    \draw  [draw opacity=0] (117.55,110.67) .. controls (118.01,105.71) and (122.4,101.79) .. (127.81,101.59) -- (128.24,111.53) -- cycle ; \draw   (117.55,110.67) .. controls (118.01,105.71) and (122.4,101.79) .. (127.81,101.59) ;  
    \draw  [draw opacity=0][dash pattern={on 0.84pt off 2.51pt}] (146.85,96.21) .. controls (146.89,96.21) and (146.94,96.21) .. (146.98,96.21) .. controls (156.09,96.22) and (163.51,102.75) .. (163.84,110.92) -- (146.96,111.49) -- cycle ; \draw  [dash pattern={on 0.84pt off 2.51pt}] (146.85,96.21) .. controls (146.89,96.21) and (146.94,96.21) .. (146.98,96.21) .. controls (156.09,96.22) and (163.51,102.75) .. (163.84,110.92) ;  
    \draw [line width=0.75]    (146.03,91.37) -- (146.14,100.14) ;
    \draw  [draw opacity=0] (145.58,101.14) .. controls (151.68,101.36) and (156.6,105.46) .. (156.99,110.61) -- (145.08,111.26) -- cycle ; \draw   (145.58,101.14) .. controls (151.68,101.36) and (156.6,105.46) .. (156.99,110.61) ;  
    \draw  [dash pattern={on 0.84pt off 2.51pt}] (290.82,172.61) -- (384.82,172.61) -- (384.82,264.61) -- (290.82,264.61) -- cycle ;
    \draw [line width=0.75]    (306.47,174.39) -- (306.54,179.68) ;
    \draw    (294.06,231.18) -- (293.62,241.82) ;
    \draw    (303.85,241.81) -- (304.96,253.52) ;
    \draw  [dash pattern={on 0.84pt off 2.51pt}]  (298.76,184.5) -- (298.76,253.51) ;
    \draw  [dash pattern={on 0.84pt off 2.51pt}]  (295.65,190.4) -- (302.36,190.32) ;
    \draw    (301.59,184.6) -- (302.13,190.39) ;
    \draw  [dash pattern={on 0.84pt off 2.51pt}]  (296.55,184.56) -- (306.78,184.33) ;
    \draw    (293.62,241.82) -- (303.95,241.81) ;
    \draw    (293.62,241.82) -- (293.3,247.25) -- (292.93,253.65) ;
    \draw    (294.06,231.18) -- (303.3,231.18) ;
    \draw    (303.3,231.18) -- (303.85,241.81) ;
    \draw    (296.41,184.57) -- (295.93,190.63) ;
    \draw  [draw opacity=0] (296.41,184.55) .. controls (296.41,184.46) and (296.4,184.38) .. (296.4,184.29) .. controls (296.41,178.93) and (300.88,174.56) .. (306.45,174.39) -- (306.78,184.3) -- cycle ; \draw   (296.41,184.55) .. controls (296.41,184.46) and (296.4,184.38) .. (296.4,184.29) .. controls (296.41,178.93) and (300.88,174.56) .. (306.45,174.39) ;  
    \draw  [draw opacity=0][dash pattern={on 0.84pt off 2.51pt}] (298.76,184.53) .. controls (298.99,180.75) and (302.09,177.71) .. (305.97,177.46) -- (306.49,184.97) -- cycle ; \draw  [dash pattern={on 0.84pt off 2.51pt}] (298.76,184.53) .. controls (298.99,180.75) and (302.09,177.71) .. (305.97,177.46) ;  
    \draw  [draw opacity=0] (301.59,184.6) .. controls (301.59,184.56) and (301.59,184.52) .. (301.59,184.48) .. controls (301.59,181.8) and (303.76,179.62) .. (306.45,179.59) -- (306.5,184.49) -- cycle ; \draw   (301.59,184.6) .. controls (301.59,184.56) and (301.59,184.52) .. (301.59,184.48) .. controls (301.59,181.8) and (303.76,179.62) .. (306.45,179.59) ;  
    \draw  [dash pattern={on 0.84pt off 2.51pt}]  (322.48,184.23) -- (322.5,246.38) ;
    \draw    (319.66,184.47) -- (319.93,188.77) ;
    \draw    (325.4,184.29) -- (325.25,188.68) ;
    \draw    (319.93,188.77) -- (320,192.77) ;
    \draw    (325.25,188.68) -- (325.25,192.77) ;
    \draw  [draw opacity=0][dash pattern={on 0.84pt off 2.51pt}] (333.53,247.11) .. controls (333.39,249.94) and (330.97,252.2) .. (328,252.19) .. controls (324.95,252.19) and (322.48,249.8) .. (322.48,246.85) .. controls (322.48,246.7) and (322.49,246.56) .. (322.5,246.41) -- (328.01,246.85) -- cycle ; \draw  [dash pattern={on 0.84pt off 2.51pt}] (333.53,247.11) .. controls (333.39,249.94) and (330.97,252.2) .. (328,252.19) .. controls (324.95,252.19) and (322.48,249.8) .. (322.48,246.85) .. controls (322.48,246.7) and (322.49,246.56) .. (322.5,246.41) ;  
    \draw    (292.93,253.65) -- (304.96,253.52) ;
    \draw    (295.65,190.4) -- (302.13,190.39) ;
    \draw    (319.67,188.86) -- (325.51,188.86) ;
    \draw    (320,192.77) -- (325.25,192.77) ;
    \draw    (319.66,184.47) -- (325.49,184.47) ;
    \draw    (296.55,184.56) -- (301.92,184.56) ;
    \draw  [dash pattern={on 0.84pt off 2.51pt}]  (336.1,217.99) -- (355.34,217.99) ;
    \draw  [draw opacity=0] (314.62,174.57) .. controls (314.77,174.56) and (314.93,174.56) .. (315.09,174.56) .. controls (320.8,174.56) and (325.43,178.98) .. (325.45,184.44) -- (315.08,184.47) -- cycle ; \draw   (314.62,174.57) .. controls (314.77,174.56) and (314.93,174.56) .. (315.09,174.56) .. controls (320.8,174.56) and (325.43,178.98) .. (325.45,184.44) ;  
    \draw  [draw opacity=0] (306.45,179.59) .. controls (306.45,179.59) and (306.45,179.59) .. (306.45,179.59) -- (306.5,184.49) -- cycle ; \draw   (306.45,179.59) .. controls (306.45,179.59) and (306.45,179.59) .. (306.45,179.59) ;  
    \draw  [draw opacity=0][dash pattern={on 0.84pt off 2.51pt}] (315.02,176.95) .. controls (315.04,176.95) and (315.06,176.95) .. (315.08,176.95) .. controls (319.26,176.95) and (322.67,180.18) .. (322.81,184.21) -- (315.08,184.47) -- cycle ; \draw  [dash pattern={on 0.84pt off 2.51pt}] (315.02,176.95) .. controls (315.04,176.95) and (315.06,176.95) .. (315.08,176.95) .. controls (319.26,176.95) and (322.67,180.18) .. (322.81,184.21) ;  
    \draw  [dash pattern={on 0.84pt off 2.51pt}]  (308.74,176.95) -- (312.41,176.91) ;
    \draw [line width=0.75]    (314.65,174.56) -- (314.7,178.88) ;
    \draw  [draw opacity=0] (314.46,179.38) .. controls (317.26,179.49) and (319.5,181.52) .. (319.67,184.06) -- (314.22,184.36) -- cycle ; \draw   (314.46,179.38) .. controls (317.26,179.49) and (319.5,181.52) .. (319.67,184.06) ;  
    \draw  [draw opacity=0][dash pattern={on 0.84pt off 2.51pt}] (298.77,253.97) .. controls (298.57,257.68) and (295.3,260.62) .. (291.29,260.62) .. controls (287.16,260.62) and (283.82,257.48) .. (283.82,253.61) .. controls (283.82,253.42) and (283.83,253.22) .. (283.85,253.03) -- (291.3,253.62) -- cycle ; \draw  [dash pattern={on 0.84pt off 2.51pt}] (298.77,253.97) .. controls (298.57,257.68) and (295.3,260.62) .. (291.29,260.62) .. controls (287.16,260.62) and (283.82,257.48) .. (283.82,253.61) .. controls (283.82,253.42) and (283.83,253.22) .. (283.85,253.03) ;  
    \draw  [draw opacity=0][dash pattern={on 0.84pt off 2.51pt}] (110.63,253.01) .. controls (110.03,258.51) and (102.91,262.86) .. (94.23,262.85) .. controls (85.15,262.84) and (77.8,258.08) .. (77.81,252.22) .. controls (77.81,251.79) and (77.85,251.36) .. (77.93,250.94) -- (94.24,252.24) -- cycle ; \draw  [dash pattern={on 0.84pt off 2.51pt}] (110.63,253.01) .. controls (110.03,258.51) and (102.91,262.86) .. (94.23,262.85) .. controls (85.15,262.84) and (77.8,258.08) .. (77.81,252.22) .. controls (77.81,251.79) and (77.85,251.36) .. (77.93,250.94) ;  
    \draw  [dash pattern={on 0.84pt off 2.51pt}]  (283.82,108.61) -- (283.85,253.03) ;
    \draw  [draw opacity=0][dash pattern={on 0.84pt off 2.51pt}] (259.71,109.6) .. controls (259.7,109.48) and (259.7,109.36) .. (259.7,109.25) .. controls (259.71,103.25) and (265.12,98.4) .. (271.78,98.41) .. controls (278.2,98.42) and (283.44,102.93) .. (283.82,108.61) -- (271.77,109.26) -- cycle ; \draw  [dash pattern={on 0.84pt off 2.51pt}] (259.71,109.6) .. controls (259.7,109.48) and (259.7,109.36) .. (259.7,109.25) .. controls (259.71,103.25) and (265.12,98.4) .. (271.78,98.41) .. controls (278.2,98.42) and (283.44,102.93) .. (283.82,108.61) ;  
    
    \draw (168,53) node [anchor=north west][inner sep=0.75pt]   [align=left] {$Q_n$};
    \draw (324,144) node [anchor=north west][inner sep=0.75pt]   [align=left] {$Q_{n + 1}$};
    \draw (391,188) node [anchor=north west][inner sep=0.75pt]   [align=left] {$Q_{n + 2}$};
    \draw (55,161) node [anchor=north west][inner sep=0.75pt]  [font=\footnotesize] [align=left] {$8c_Ml_n$};
    \draw (240,211) node [anchor=north west][inner sep=0.75pt]  [font=\footnotesize] [align=left] {$8c_Ml_{n + 1}$};
    \draw (338,239) node [anchor=north west][inner sep=0.75pt]  [font=\footnotesize] [align=left] {$8c_Ml_{n + 2}$};
    \draw (178,267) node [anchor=north west][inner sep=0.75pt]  [font=\footnotesize] [align=left] {$w_n$};
    \draw (330,267) node [anchor=north west][inner sep=0.75pt]  [font=\footnotesize] [align=left] {$w_{n + 1}$};
    \draw (401,267) node [anchor=north west][inner sep=0.75pt]  [font=\footnotesize] [align=left] {$w_{n + 2}$};
    \draw (68,233.97) node [anchor=north west][inner sep=0.75pt]  [font=\scriptsize] [align=left] {{\scriptsize $R_{m_1}$}};
    \draw (55,212.97) node [anchor=north west][inner sep=0.75pt]  [font=\scriptsize] [align=left] {{\scriptsize $R_{m_1 + 1}$}};
    \draw (55,111.97) node [anchor=north west][inner sep=0.75pt]  [font=\scriptsize] [align=left] {{\scriptsize $R_{m_2 - 1}$}};
    \draw (178.12,107.97) node [anchor=north west][inner sep=0.75pt]  [font=\scriptsize] [align=left] {{\scriptsize $R_{m_2 + s_1}$}};
    \draw (178.18,120.62) node [anchor=north west][inner sep=0.75pt]  [font=\scriptsize] [align=left] {{\scriptsize $R_{m_2 + s_1 + 1}$}};
    \draw (129,92) node [anchor=north west][inner sep=0.75pt]   [align=left] {{\tiny $\cdots$}};
    
    \end{tikzpicture}
    \caption{The folding domain}
\end{figure}

    \medskip
    
\begin{figure}[h]
        \centering    
        \begin{tikzpicture}[x=0.75pt,y=0.75pt,yscale=-1,xscale=1]
        
        \draw  [dash pattern={on 0.84pt off 2.51pt}]  (51,224.4) -- (345.23,224.4) ;
        \draw  [dash pattern={on 0.84pt off 2.51pt}]  (58.72,100.47) -- (58.72,224.4) ;
        \draw  [dash pattern={on 0.84pt off 2.51pt}]  (329.77,103) -- (329.77,224.4) ;
        \draw  [dash pattern={on 0.84pt off 2.51pt}]  (58.72,224.4) -- (58.72,240.42) ;
        \draw  [dash pattern={on 0.84pt off 2.51pt}]  (329.77,224.4) -- (329.77,240.42) ;
        \draw  [dash pattern={on 0.84pt off 2.51pt}]  (214,233.67) -- (326.77,233.67) ;
        \draw [shift={(328.77,233.67)}, rotate = 180] [color={rgb, 255:red, 0; green, 0; blue, 0 }  ][line width=0.75]    (10.93,-3.29) .. controls (6.95,-1.4) and (3.31,-0.3) .. (0,0) .. controls (3.31,0.3) and (6.95,1.4) .. (10.93,3.29)   ;
        \draw    (188.13,216.17) -- (167.09,216.17) ;
        \draw [shift={(165.09,216.17)}, rotate = 360] [color={rgb, 255:red, 0; green, 0; blue, 0 }  ][line width=0.75]    (10.93,-3.29) .. controls (6.95,-1.4) and (3.31,-0.3) .. (0,0) .. controls (3.31,0.3) and (6.95,1.4) .. (10.93,3.29)   ;
        \draw  [dash pattern={on 0.84pt off 2.51pt}]  (165.09,208.15) -- (165.09,224.19) ;
        \draw  [dash pattern={on 0.84pt off 2.51pt}]  (218.68,208.15) -- (218.68,224.19) ;
        \draw  [dash pattern={on 0.84pt off 2.51pt}]  (190.81,100.47) -- (190.81,316.62) ;
        \draw    (194.56,216.17) -- (216.68,216.17) ;
        \draw [shift={(218.68,216.17)}, rotate = 180] [color={rgb, 255:red, 0; green, 0; blue, 0 }  ][line width=0.75]    (10.93,-3.29) .. controls (6.95,-1.4) and (3.31,-0.3) .. (0,0) .. controls (3.31,0.3) and (6.95,1.4) .. (10.93,3.29)   ;
        \draw  [dash pattern={on 0.84pt off 2.51pt}]  (169.37,233.67) -- (60.72,233.67) ;
        \draw [shift={(58.72,233.67)}, rotate = 360] [color={rgb, 255:red, 0; green, 0; blue, 0 }  ][line width=0.75]    (10.93,-3.29) .. controls (6.95,-1.4) and (3.31,-0.3) .. (0,0) .. controls (3.31,0.3) and (6.95,1.4) .. (10.93,3.29)   ;
        \draw  [draw opacity=0] (164.89,223.98) .. controls (164.84,222.85) and (164.82,221.72) .. (164.82,220.58) .. controls (164.82,172.93) and (209.76,134.23) .. (265.48,133.62) -- (266.82,220.58) -- cycle ; \draw   (164.89,223.98) .. controls (164.84,222.85) and (164.82,221.72) .. (164.82,220.58) .. controls (164.82,172.93) and (209.76,134.23) .. (265.48,133.62) ;  
        \draw  [draw opacity=0][dash pattern={on 0.84pt off 2.51pt}] (190.81,224.7) .. controls (190.81,224.45) and (190.81,224.2) .. (190.81,223.95) .. controls (190.81,187.43) and (223.92,157.79) .. (264.86,157.56) -- (265.35,223.95) -- cycle ; \draw  [dash pattern={on 0.84pt off 2.51pt}] (190.81,224.7) .. controls (190.81,224.45) and (190.81,224.2) .. (190.81,223.95) .. controls (190.81,187.43) and (223.92,157.79) .. (264.86,157.56) ;  
        \draw  [draw opacity=0] (219.02,224.47) .. controls (220.19,200.18) and (240.37,180.7) .. (265.46,179.74) -- (267.42,226.69) -- cycle ; \draw   (219.02,224.47) .. controls (220.19,200.18) and (240.37,180.7) .. (265.46,179.74) ;  
        \draw    (264.48,133.62) -- (264.9,179.71) ;
        \draw  [dash pattern={on 0.84pt off 2.51pt}]  (265.61,179.71) -- (266.1,222.73) ;
        \draw    (246.27,216.17) -- (265.38,216.17) ;
        \draw [shift={(267.38,216.17)}, rotate = 180] [color={rgb, 255:red, 0; green, 0; blue, 0 }  ][line width=0.75]    (10.93,-3.29) .. controls (6.95,-1.4) and (3.31,-0.3) .. (0,0) .. controls (3.31,0.3) and (6.95,1.4) .. (10.93,3.29)   ;
        \draw    (238.37,216.17) -- (220.33,216.17) ;
        \draw [shift={(218.33,216.17)}, rotate = 360] [color={rgb, 255:red, 0; green, 0; blue, 0 }  ][line width=0.75]    (10.93,-3.29) .. controls (6.95,-1.4) and (3.31,-0.3) .. (0,0) .. controls (3.31,0.3) and (6.95,1.4) .. (10.93,3.29)   ;
        \draw  [dash pattern={on 0.84pt off 2.51pt}]  (264.86,157.56) -- (388.35,157.35) ;
        \draw    (159,261.14) -- (225,261.14) ;
        \draw    (165.09,224.19) -- (159,261.14) ;
        \draw    (225,261.14) -- (218.68,224.19) ;
        \draw    (150.43,302.71) -- (236.43,302.71) ;
        \draw    (150.43,302.71) -- (159,261.14) ;
        \draw    (236.43,302.71) -- (225,261.14) ;
        \draw    (152.56,249.17) -- (174.68,249.17) ;
        \draw [shift={(176.68,249.17)}, rotate = 180] [fill={rgb, 255:red, 0; green, 0; blue, 0 }  ][line width=0.08]  [draw opacity=0] (12,-3) -- (0,0) -- (12,3) -- cycle    ;
        \draw    (147.56,278.17) -- (169.68,278.17) ;
        \draw [shift={(171.68,278.17)}, rotate = 180] [fill={rgb, 255:red, 0; green, 0; blue, 0 }  ][line width=0.08]  [draw opacity=0] (12,-3) -- (0,0) -- (12,3) -- cycle    ;
        \draw    (286.43,138.71) -- (286.43,173.71) ;
        \draw    (264.48,133.62) -- (286.43,138.71) ;
        \draw    (264.9,179.71) -- (286.43,173.71) ;
        \draw    (304.43,142.71) -- (304.43,169.71) ;
        \draw    (286.43,138.71) -- (304.43,142.71) ;
        \draw    (286.43,173.71) -- (304.43,169.71) ;
        \draw    (332.53,150.26) -- (332.53,162.27) ;
        \draw    (325.1,148.48) -- (332.53,150.26) ;
        \draw    (325.1,164.05) -- (332.53,162.27) ;
        \draw    (325.1,148.48) -- (325.1,164.05) ;
        \draw  [draw opacity=0] (332.76,150.36) .. controls (346.54,150.44) and (357.69,161.32) .. (357.69,174.74) .. controls (357.69,174.88) and (357.68,175.02) .. (357.68,175.16) -- (332.62,174.74) -- cycle ; \draw   (332.76,150.36) .. controls (346.54,150.44) and (357.69,161.32) .. (357.69,174.74) .. controls (357.69,174.88) and (357.68,175.02) .. (357.68,175.16) ;  
        \draw    (346.88,175.51) -- (348.07,184.43) ;
        \draw  [draw opacity=0] (332.53,162.27) .. controls (340.46,162.31) and (346.88,168.19) .. (346.88,175.44) .. controls (346.88,175.46) and (346.88,175.49) .. (346.88,175.51) -- (332.45,175.44) -- cycle ; \draw   (332.53,162.27) .. controls (340.46,162.31) and (346.88,168.19) .. (346.88,175.44) .. controls (346.88,175.46) and (346.88,175.49) .. (346.88,175.51) ;  
        \draw    (358.19,175.36) -- (346.88,175.51) ;
        \draw  [draw opacity=0][dash pattern={on 0.84pt off 2.51pt}] (333.12,156.77) .. controls (343.41,156.84) and (351.82,165.12) .. (352.57,175.62) -- (333,177.15) -- cycle ; \draw  [dash pattern={on 0.84pt off 2.51pt}] (333.12,156.77) .. controls (343.41,156.84) and (351.82,165.12) .. (352.57,175.62) ;  
        \draw  [dash pattern={on 0.84pt off 2.51pt}]  (352.52,133.57) -- (352.52,254.97) ;
        \draw    (357.07,184.43) -- (348.07,184.43) ;
        \draw    (357.68,175.16) -- (357.07,184.43) ;
        \draw    (356.07,192.93) -- (349.07,192.93) ;
        \draw    (357.07,184.43) -- (356.07,192.93) ;
        \draw    (348.07,184.43) -- (349.07,192.93) ;
        \draw    (165.09,224.19) -- (218.68,224.19) ;
        
        \draw (173.29,227.23) node [anchor=north west][inner sep=0.75pt]  [font=\scriptsize] [align=left] {{\scriptsize $2c_M a_{m_1}$}};
        \draw (172.03,199.43) node [anchor=north west][inner sep=0.75pt]  [font=\scriptsize] [align=left] {{\scriptsize $2c_M a_{m_2}$}};
        \draw (104,271.29) node [anchor=north west][inner sep=0.75pt]  [font=\footnotesize] [align=left] {$R_{m_2 - 2}$};
        \draw (111,242.29) node [anchor=north west][inner sep=0.75pt]  [font=\footnotesize] [align=left] {$R_{m_2 - 1}$};
        \draw (308,138.95) node [anchor=north west][inner sep=0.75pt]  [font=\normalsize] [align=left] {{\scriptsize $\cdots$}};
        \draw (267.43,116.25) node [anchor=north west][inner sep=0.75pt]  [font=\footnotesize] [align=left] {{\scriptsize $R_{m_2}$}};
        \draw (287.43,123.75) node [anchor=north west][inner sep=0.75pt]  [font=\footnotesize] [align=left] {{\scriptsize $R_{m_2 + 1}$}};
        \draw (316.43,131.75) node [anchor=north west][inner sep=0.75pt]  [font=\footnotesize] [align=left] {{\scriptsize $R_{m_2 + s_1 - 1}$}};
        \draw (361.43,170.75) node [anchor=north west][inner sep=0.75pt]  [font=\footnotesize] [align=left] {{\scriptsize $R_{m_2 + s_1}$}};
        \draw (359.07,187.43) node [anchor=north west][inner sep=0.75pt]  [font=\footnotesize] [align=left] {{\scriptsize $R_{m_2 + s_1 + 1}$}};
        \draw (224.03,200.43) node [anchor=north west][inner sep=0.75pt]  [font=\scriptsize] [align=left] {{\scriptsize $2c_M a_{m_2}$}};
        \end{tikzpicture}
        \caption{The twisting part, and the choice of $s_1$}
    \end{figure}
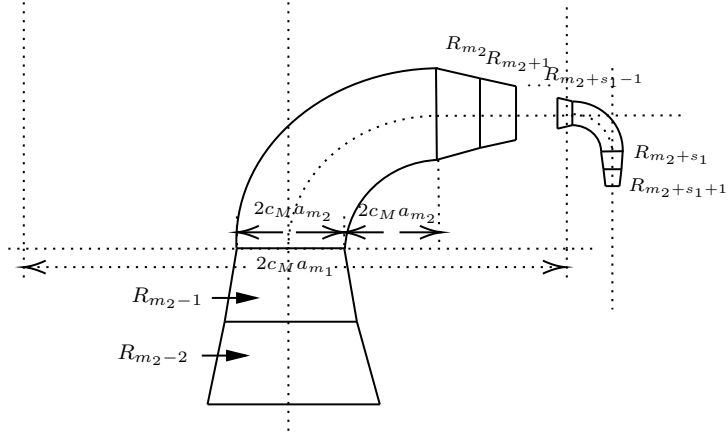
    The twisting part is a quarter of an annulus with radii $4c_M a_{m_2}$ and $2c_M a_{m_2}$
    (see the following page for the corresponding figure).
    The choice of $m_d$ ensures that every vertical ``pipe" has sufficient length,
    and the integer $s_d$ guarantees that adjacent ``pipes" do not overlap. 

\medskip

For the estimation of $w_n$, we may assume, for convenience, that $s_d = 0$ for all $d = 1, \ldots, K_n - 1$, and that $m_{K_n}$ satisfies \eqref{equ:enoughlength}.
Thus, the sequence $\{m_d\}^{K_n}_{d = 1}$ satisfies that
$$i_n + 1 = m_1 < m_2 < \ldots < m_{K_n} = i_{n + 1},$$ 
and,
for all $d\in\{1,2,\ldots, K_n - 1\}$,
\begin{equation}\label{equ:sequence}
    \frac{l_n}{2} \leq a_{m_d} + a_{m_d + 1} + \cdots + a_{m_{d + 1}} \leq 2l_n.
\end{equation}
Since the sequence $\{a_n\}$ is decreasing, and by \eqref{equ:sequence} and the construction process, we have
$$w_n \lesssim c_M\sum^{K_n}_{d = 1} a_{m_d} \lesssim a_{m_1} + \sum^{K_n - 1}_{d = 1} \frac{2l_n}{m_{d + 1} - m_d}.$$
By \eqref{equ:sequence} and the Cauchy--Schwartz inequality, we conclude that
\begin{align*}
    \frac{(l_n)^2}{4} 
    & \leq \left(\sum^{m_{d + 1} - 1}_{k = m_d}a_k\right)^2
      \leq (m_{d + 1} - m_d) \sum^{m_{d + 1} - 1}_{k = m_d}(a_k)^2
\end{align*}
Thus, by \eqref{equ:category},
$$\sum^{K_n - 1}_{d = 1} \frac{l_n}{m_{d + 1} - m_d} \lesssim \frac{1}{l_n}\sum^{i_{n + 1}}_{k = i_n + 1} (a_k)^2 \sim l_n.$$
Notice that $a_{m_1} \lesssim l_n$, this shows that $w_n\lesssim l_n$.


The newly constructed domain, denoted by $\Oz$,
is a Jordan domain since $\sum_{n\in\bN^+}w_n < \fz$.
Note that $\diam_I(\Oz) = \fz$. 
By Theorem \ref{main-3}, we can find a parametrization $\vz$
such that it does not admit a homeomorpic $W^{1,1}$ extension.
Moreover, by \eqref{equ:quasihyper:intergrability} and \cite[Lemma 3.2]{BJKXZ2024}, for some $z_0\in\Oz$,
$$\int_{\Oz} \phi(h_{\Oz}(z_0,z))\,\od z < \fz.$$
This completes the proof for Theorem \ref{main-2}.
\end{proof}

\section{Acknowledgments}
This research was supported by the Finnish centre of excellence in Random-
ness and Structures of the Academy of Finland, funding decision number: 346305.
The author would like to thank Prof. Pekka Koskela for his valuable advice and insightful suggestions on this research.


\begin{thebibliography}{99}

    \bibitem{A1995}
    S.~S.~Antman,
    Nonlinear problems of elasticity, Applied Mathematical Sciences,
    vol. 107, Springer-Verlag, New York, 1995.

    \bibitem{AIM2009}
    K. Astala, T. Iwaniec, G. Martin,
    Elliptic Partial Differential Equations and Quasiconformal Mappings in the Plane,
    Princeton University Press, 2009.

    \bibitem{B1976}
    J.~M.~Ball, 
    Convexity conditions and existence theorems in nonlinear elasticity,
    Arch. Rational Mech. Anal. 63 (1976/77), no. 4, 337-403.

    \bibitem{BJKXZ2024}
    O.~Bouchala, J.~J\"a\"askel\"ainen, P.~Koskela, H.~Xu, X.~Zhou,
    Homeomorphic Sobolev extensions of parametrizations of Jordan curves,
    J. Funct. Anal. 228 (4) (2025) 110721, 22pp.

    \bibitem{C1988}
    P.~G.~Ciarlet,
    Mathematical elasticity Vol. I. Three-dimensional elasticity,
    Studies in Mathematics and its Applications, vol. 20. North-Holland Publishing Co., Amsterdam, 1988.

    \bibitem{HK2014}
    S.~Hencl, P.~Koskela,
    Lectures on mappings of finite distortion.
    Lecture Notes in Mathematics, 2096. Springer, Cham, 2014.

    \bibitem{GH1962}
    F.~W.~Gehring, W.~K.~Hayman,
    An inequality in the theory of conformal mapping.
    J. Math. Pure. Appl. 41(9) (1962), 353-361.

    \bibitem{IM2001}
    T.~Iwaniec, G.~Martin,
    Geometric Function Theory and Non-linear Analysis,
    Oxford Mathematical Monographs, Oxford University Press, 2001.

    \bibitem{KKO2020}
    P.~Koskela, A.~Koski, J.~Onninen,
    Sobolev homeomorphic extensions onto John domains,
    J. Funct. Anal. 279 (2020), no. 10, 108719, 17 pp.

    \bibitem{KO2021}
    A.~Koski, J.~Onninen,
    Sobolev homeomorphic extensions,
    J. Eur. Math. Soc. 23 (2021), no. 12, 4065-4089.

    \bibitem{KO2023}
    A.~Koski, J.~Onninen, 
    The Sobolev Jordan--Sch\"onflies problem, 
    Adv. Math. 413 (2023) 108795.

    \bibitem{R1989}
    Yu.~G.~Reshetnyak,
    Space mappings with bounded distortion,
    American Mathematical Society, Providence, RI, 1989.

    \bibitem{V2007}
    G.~C.~Verchota,
    Harmonic homeomorphisms of the closed disc to itself need be in $W^{1,p}$, $p < 2$, but not $W^{1,2}$,
    Proc. Amer. Math. Soc. (2007), vol. 135, no. 3, 891--894.

    \bibitem{Z2019}
    Y.~R.-Y.~Zhang,
    Schoenflies solutions with conformal boundary values may fail to be Sobolev,
    Ann. Acad. Sci. Fenn. Ser. A I Math. (2019), vol. 44, no. 2, 791-796.

    \end{thebibliography}
\end{document}